\documentclass[12pt,english,reqno]{amsart}
\usepackage{charter}
\usepackage[T1]{fontenc}
\usepackage[latin9]{inputenc}
\usepackage{geometry}
\usepackage{color}
\usepackage{babel}
\usepackage{prettyref}
\usepackage{mathrsfs}
\usepackage{amsthm}
\usepackage{amstext}
\usepackage{amssymb}
\usepackage{setspace}
\usepackage{esint}
\usepackage[numbers]{natbib}
\usepackage[all]{xy}
\usepackage[unicode=true,
 bookmarks=false,
 breaklinks=false,pdfborder={0 0 1},backref=page,colorlinks=true]
 {hyperref}
\hypersetup{pdftitle={On positive definiteness over locally compact quantum groups},
 pdfauthor={Volker Runde and Ami Viselter},
 pdfstartview={XYZ null null 1.25}}

\makeatletter
\numberwithin{equation}{section}
\theoremstyle{plain}
\newtheorem{thm}{\protect\theoremname}[section]
  \theoremstyle{definition}
  \newtheorem{defn}[thm]{\protect\definitionname}
  \theoremstyle{definition}
  \newtheorem{example}[thm]{\protect\examplename}
  \theoremstyle{plain}
  \newtheorem{lem}[thm]{\protect\lemmaname}
  \theoremstyle{remark}
  \newtheorem{rem}[thm]{\protect\remarkname}
  \theoremstyle{plain}
  \newtheorem{prop}[thm]{\protect\propositionname}
  \theoremstyle{plain}
  \newtheorem{cor}[thm]{\protect\corollaryname}

\usepackage{eucal}
\let\mathcal=\CMcal
\usepackage{dsfont}
\usepackage{graphicx}
\usepackage{mathtools} % for \prescript

\let\shorti=\i
\let\ldash=\l

\newrefformat{eq}{\eqref{#1}}
\newrefformat{def}{Definition \ref{#1}}
\newrefformat{lem}{Lemma \ref{#1}}
\newrefformat{pro}{Proposition \ref{#1}}
\newrefformat{prop}{Proposition \ref{#1}}
\newrefformat{cor}{Corollary \ref{#1}}
\newrefformat{thm}{Theorem \ref{#1}}
\newrefformat{exa}{Example \ref{#1}}
\newrefformat{rem}{Remark \ref{#1}}
\newrefformat{sec}{Section \ref{#1}}
\newrefformat{sub}{Subsection \ref{#1}}
\newrefformat{conj}{Conjecture \ref{#1}}

\makeatother

  \providecommand{\corollaryname}{Corollary}
  \providecommand{\definitionname}{Definition}
  \providecommand{\examplename}{Example}
  \providecommand{\lemmaname}{Lemma}
  \providecommand{\propositionname}{Proposition}
  \providecommand{\remarkname}{Remark}
\providecommand{\theoremname}{Theorem}

\begin{document}
\global\long\def\e{\varepsilon}
\global\long\def\N{\mathbb{N}}
\global\long\def\Z{\mathbb{Z}}
\global\long\def\Q{\mathbb{Q}}
\global\long\def\R{\mathbb{R}}
\global\long\def\C{\mathbb{C}}
\global\long\def\G{\mathbb{G}}
\global\long\def\HH{\mathbb{H}}

\global\long\def\H{\EuScript H}
\global\long\def\K{\EuScript K}
\global\long\def\a{\alpha}
\global\long\def\be{\beta}
\global\long\def\l{\lambda}
\global\long\def\om{\omega}
\global\long\def\z{\zeta}
\global\long\def\Aa{\mathcal{A}}

\global\long\def\Ree{\operatorname{Re}}
\global\long\def\Img{\operatorname{Im}}
\global\long\def\linspan{\operatorname{span}}
\global\long\def\slim{\operatorname*{s-lim}}
\global\long\def\clinspan{\operatorname{\overline{span}}}
\global\long\def\presb#1#2{\prescript{}{#1}{#2}}

\global\long\def\tensor{\otimes}
\global\long\def\tensormin{\otimes_{\mathrm{min}}}
\global\long\def\tensorn{\overline{\otimes}}

\global\long\def\A{\forall}

\global\long\def\i{\mathrm{id}}

\global\long\def\one{\mathds{1}}
\global\long\def\tr{\mathrm{tr}}
\global\long\def\Ww{\mathds{W}}
\global\long\def\wW{\text{\reflectbox{\ensuremath{\Ww}}}\:\!}
\global\long\def\op{\mathrm{op}}
\global\long\def\WW{{\mathds{V}\!\!\text{\reflectbox{\ensuremath{\mathds{V}}}}}}

\global\long\def\Linfty#1{L^{\infty}(#1)}
\global\long\def\Lone#1{L^{1}(#1)}
\global\long\def\LoneStar#1{L_{*}^{1}(#1)}
\global\long\def\Ltwo#1{L^{2}(#1)}
\global\long\def\Cz#1{C_{0}(#1)}
\global\long\def\CzU#1{C_{0}^{\mathrm{u}}(#1)}
\global\long\def\Ad#1{\mathrm{Ad}(#1)}
\global\long\def\VN#1{\mathrm{VN}(#1)}
\global\long\def\d{~\mathrm{d}}

\title{On positive definiteness over locally compact quantum groups}

\author{Volker Runde}

\address{Department of Mathematical and Statistical Sciences, University of
Alberta, Edmonton, Alberta T6G 2G1, Canada}

\email{vrunde@ualberta.ca}

\author{Ami Viselter}

\address{Department of Mathematical and Statistical Sciences, University of
Alberta, Edmonton, Alberta T6G 2G1, Canada}

\address{Department of Mathematics, University of Haifa, 31905 Haifa, Israel}

\email{aviselter@staff.haifa.ac.il}

\thanks{Both authors were supported by NSERC Discovery Grants.}


\begin{abstract}
The notion of positive-definite functions over locally compact quantum
groups was recently introduced and studied by Daws and Salmi. Based
on this work, we generalize various well-known results about positive-definite
functions over groups to the quantum framework. Among these are theorems
on ``square roots'' of positive-definite functions, comparison of
various topologies, positive-definite measures and characterizations
of amenability, and the separation property with respect to compact
quantum subgroups.
\end{abstract}

\maketitle

\section*{Introduction}

Positive-definite functions over locally compact groups, introduced
by Godement in \citep{Godement__positive_def_func}, play a central
role in abstract harmonic analysis. If $G$ is a locally compact group,
a continuous function $f:G\to\C$ is called \emph{positive definite}
if for every $n\in\N$ and $s_{1},\ldots,s_{n}\in G$, the matrix
$\left(f(s_{i}^{-1}s_{j})\right)_{1\leq i,j\leq n}$ is positive (we
always take continuity as part of the definition). Positive-definite
functions are tightly connected with various aspects of the group,
such as representations, group properties (amenability and other approximation
properties, property (T), etc.), the Banach algebras associated to
the group and many more, as exemplified by the numerous papers dedicated
to them. It is thus natural to extend this theory to a framework more
general than locally compact groups. This was done in the context
of Kac algebras by Enock and Schwartz \citep[Section 1.3]{Enock_Schwartz__book}.
Recently, Daws \citep{Daws__CPM_LCQGs_2012} and Daws and Salmi \citep{Daws_Salmi__CPD_LCQGs_2013}
generalized this work to the much wider context of locally compact
quantum groups in the sense of Kustermans and Vaes \citep{Kustermans_Vaes__LCQG_C_star,Kustermans_Vaes__LCQG_von_Neumann}.
They introduced several notions of positive definiteness, corresponding
to the classical ones, and established the precise relations between
them.

These foundations being laid, the next step should be generalizing
well-known useful results from abstract harmonic analysis about positive-definite
functions to locally compact quantum groups. This is the purpose of
the present paper, which is organized as follows. 

In \prettyref{sec:square_integrable_positive_def} we generalize a
result of Godement, essentially saying that a positive-definite function
has a ``square root'' if and only if it is square integrable. 

A theorem of Ra{\u\shorti}kov \citep{Raikov__types_of_conv} and
Yoshizawa \citep{Yoshizawa__convergence_PD_func} says that on the
set of positive-definite functions of norm $1$, the $w^{*}$-topology
induced by $L^{1}$ coincides with the topology of uniform convergence
on compact subsets. This result was improved by several authors, and
eventually Granirer and Leinert \citep{Granirer_Leinert__top_coincide}
generalized it to treat the different topologies on the unit sphere
of the Fourier--Stieltjes algebra. Hu, Neufang and Ruan asked in \citep{Hu_Neufang_Ruan__mod_maps_LCQG}
whether this result extends to locally compact quantum groups. We
give an affirmative answer to their question in \prettyref{sec:comparison_of_topologies}.
Generalizing other results from \citep{Granirer_Leinert__top_coincide}
as well, we require the theory of noncommutative $L^{p}$-spaces of
locally compact quantum groups. The background on this subject appears
in \prettyref{sec:convolution_L_p}.

Another notion due to Godement is that of positive-definite measures.
He established an important connection between these and amenability
of the group in question. In \prettyref{sec:char_co_amen_of_dual}
we extend this result to locally compact quantum groups. 

The separation property of locally compact groups with respect to
closed subgroups was introduced by Lau and Losert \citep{Lau_Losert__w_star_closed_comp_inv_subsp}
and Kaniuth and Lau \citep{Kaniuth_Lau__separation_prop}, and was
subsequently studied by several authors. A fundamental result is that
the separation property is always satisfied with respect to compact
subgroups. \prettyref{sec:separation_property} is devoted to generalizing
this to locally compact quantum groups. We introduce the separation
property with respect to closed quantum subgroups, find a condition
under which the separation property is satisfied with respect to a
given compact quantum subgroup, and show that it is indeed satisfied
in many examples, including $\mathbb{T}$ as a closed quantum subgroup
of quantum $E(2)$.

We remark that most sections are independent of each other, but results
from \prettyref{sec:comparison_of_topologies} are needed in other
sections.

\section{Preliminaries}

We begin with fixing some conventions. Given a Hilbert space $\H$
and vectors $\z,\eta\in\H$, we denote by $\om_{\z,\eta}$ the functional
that takes $x\in B(\H)$ to $\left\langle x\z,\eta\right\rangle $,
and let $\om_{\z}:=\om_{\z,\z}$. The identity map on a $C^{*}$-algebra
$A$ is denoted by $\i$, and its unit, if exists, by $\one$. For
a functional $\om\in A^{*}$, we define $\overline{\om}\in A^{*}$
by $\overline{\om}(x):=\overline{\om(x^{*})}$, $x\in A$. When no
confusion is caused, we also write $\om$ for its unique extension
to the multiplier algebra $M(A)$ that is strictly continuous on the
closed unit ball of $M(A)$ \citep[Corollary 5.7]{Lance}. 

Let $A,B$ be $C^{*}$-algebras. A $*$-homomorphism from $A$ to
$B$ or, more generally, to $M(B)$ that is nondegenerate (namely,
$\linspan\Phi(A)B$ is dense in $B$) has a unique extension to a
(unital) $*$-homomorphism from $M(A)$ to $M(B)$ \citep[Proposition 2.1]{Lance}.
We use the same notation for this extension.

For an n.s.f.~(normal, semi-finite, faithful) weight $\varphi$ on
a von Neumann algebra $M$ \citep[Chapter VII]{Takesaki__book_vol_2},
we denote $\mathcal{N}_{\varphi}:=\left\{ x\in M:\varphi(x^{*}x)<\infty\right\} $.

The symbol $\sigma$ stands for the flip operator $x\tensor y\mapsto y\tensor x$,
for $x,y$ in some $C^{*}$-algebras. We use the symbols $\tensor,\tensorn,\tensormin$
for the Hilbert space, normal spatial and minimal tensor products,
respectively. 

The basics of positive-definite functions on locally compact groups
are presented in the book of Dixmier \citep{Dixmier__C_star_English}.
From time to time we will refer to the Banach algebras associated
with a locally compact group $G$, such as the Fourier algebra $A(G)$
and the Fourier--Stieltjes algebra $B(G)$; see Eymard \citep{Eymard__Fourier_alg}.
For the Tomita--Takesaki theory, see the books by Str\u{a}til\u{a}
\citep{Stratila__mod_thy} and Takesaki \citep{Takesaki__book_vol_2},
or Takesaki's original monograph \citep{Takesaki__Tomita_thy}. We
recommend B{\'e}dos, Murphy and Tuset \citep[Section 2]{Bedos_Murphy_Tuset__am_and_coam}
for statements and proofs of folklore facts about the slice maps at
the $C^{*}$-algebraic level.

\subsection{Locally compact quantum groups}

The following axiomatization of locally compact quantum groups is
due to Kustermans and Vaes \citep{Kustermans_Vaes__LCQG_C_star,Kustermans_Vaes__LCQG_von_Neumann}
(see also Van Daele \citep{Van_Daele__LCQGs}). It describes the same
objects as that of Masuda, Nakagami and Woronowicz \citep{Masuda_Nakagami_Woronowicz}.
Unless stated otherwise, the material in this subsection is taken
from \citep{Kustermans_Vaes__LCQG_C_star,Kustermans_Vaes__LCQG_von_Neumann}.
\begin{defn}
A \emph{locally compact quantum group} (henceforth abbreviated to
``LCQG'') is a pair $\G=(\Linfty{\G},\Delta)$ with the following
properties:
\begin{enumerate}
\item $\Linfty{\G}$ is a von Neumann algebra;
\item $\Delta:\Linfty{\G}\to\Linfty{\G}\tensorn\Linfty{\G}$ is a co-multiplication,
that is, a faithful, normal, unital $*$-homomorphism which is co-associative:
$(\Delta\tensor\i)\Delta=(\i\tensor\Delta)\Delta$;
\item there exist n.s.f.~weights $\varphi,\psi$ on $\Linfty{\G}$, called
the Haar weights, satisfying
\[
\begin{split} & \varphi((\omega\tensor\i)\Delta(x))=\omega(\one)\varphi(x)\text{ for all }\om\in\Linfty{\G}_{*}^{+},x\in\Linfty{\G}^{+}\text{ such that }\varphi(x)<\infty\text{ (left invariance),}\\
 & \psi((\i\tensor\omega)\Delta(x))=\omega(\one)\psi(x)\text{ for all }\om\in\Linfty{\G}_{*}^{+},x\in\Linfty{\G}^{+}\text{ such that }\psi(x)<\infty\text{ (right invariance).}
\end{split}
\]
 
\end{enumerate}
\end{defn}
Let $\G$ be a LCQG. The left and right Haar weights, only whose existence
is assumed, are unique up to scaling. The predual of $\Linfty{\G}$
is denoted by $\Lone{\G}$. We define a convolution $*$ on $\Lone{\G}$
by $(\om_{1}*\om_{2})(x):=(\om_{1}\tensor\om_{2})\Delta(x)$ ($\om_{1},\om_{2}\in\Lone{\G}$,
$x\in\Linfty{\G}$), making the pair $(\Lone{\G},*)$ into a Banach
algebra. We write $\Ltwo{\G}$ for the Hilbert space of the GNS construction
for $(\Linfty{\G},\varphi)$, and let $\Lambda:\mathcal{N}_{\varphi}\to\Ltwo{\G}$
stand for the canonical injection. A fundamental feature of the theory
is that of duality: $\G$ has a \emph{dual} LCQG $\hat{\G}=(\Linfty{\hat{\G}},\hat{\Delta})$.
Objects pertaining to $\hat{\G}$ will be denoted by adding a hat,
e.g.~$\hat{\varphi},\hat{\psi}$. The GNS construction for $(\Linfty{\hat{\G}},\hat{\varphi})$
yields the same Hilbert space $\Ltwo{\G}$, and henceforth we will
consider both $\Linfty{\G}$ and $\Linfty{\hat{\G}}$ as acting (standardly)
on $\Ltwo{\G}$. We write $J,\hat{J}$ for the modular conjugations
relative to $\Linfty{\G},\Linfty{\hat{\G}}$, respectively, both acting
on $\Ltwo{\G}$.
\begin{example}
Every locally compact group $G$ induces two LCQGs as follows. First,
the LCQG that is identified with $G$ is $(\Linfty G,\Delta)$, where
$(\Delta(f))(t,s):=f(ts)$ for $f\in\Linfty G$ and $t,s\in G$ using
the identification $\Linfty G\tensorn\Linfty G\cong\Linfty{G\times G}$,
and $\varphi$ and $\psi$ are integration against the left and right
Haar measures of $G$, respectively. All LCQGs whose $\Linfty{\G}$
is commutative have this form. Second, the dual of the above, which
is the LCQG $(\VN G,\Delta)$, where $\VN G$ is the left von Neumann
algebra of $G$, $\Delta$ is the unique normal $*$-homomorphism
$\VN G\to\VN G\tensorn\VN G$ mapping the translation $\l_{t}$, $t\in G$,
to $\l_{t}\tensor\l_{t}$, and $\varphi$ and $\psi$ are the Plancherel
weight on $\VN G$. The LCQGs that are co-commutative, namely whose
$\Lone{\G}$ is commutative, are precisely the ones of this form.
The $L^{2}$-Hilbert space of both LCQGs is $\Ltwo G$.
\end{example}
The \emph{left regular co-representation} of $\G$ is a unitary $W\in\Linfty{\G}\tensorn\Linfty{\hat{\G}}$
satisfying $\Delta(x)=W^{*}(\one\tensor x)W$ for every $x\in\Linfty{\G}$
and $(\Delta\tensor\i)(W)=W_{13}W_{23}$ (using leg numbering). The
left regular co-representation of $\hat{\G}$ is $\hat{W}=\sigma(W^{*})$.
The set $\Cz{\G}:=\overline{\{(\i\tensor\hat{\om})(W):\hat{\om}\in\Lone{\hat{\G}}\}}^{\left\Vert \cdot\right\Vert }$
is a weakly dense $C^{*}$-subalgebra of $\Linfty{\G}$, satisfying
$\Delta(\Cz{\G})\subseteq M(\Cz{\G}\tensormin\Cz{\G})$. This allows
to define a convolution $*$ on $\Cz{\G}^{*}$, which becomes a Banach
algebra. Viewing $\Lone{\G}$ as a subspace of $\Cz{\G}^{*}$ by restriction,
the former is a (closed, two-sided) ideal in the latter. We define
a map $\l:\Lone{\G}\to\Cz{\hat{\G}}$ by $\l(\om):=(\om\tensor\i)(W)$.
It is easily checked that $\l$ is a contractive homomorphism.

We review the construction of the left-invariant weight $\hat{\varphi}$
of $\hat{\G}$. Let $\mathcal{I}$ stand for all ``square-integrable
elements of $\Lone{\G}$'', namely all $\om\in\Lone{\G}$ such that
there is $M<\infty$ with $\left|\om(x^{*})\right|\leq M\left\Vert \Lambda(x)\right\Vert $
for every $x\in\mathcal{N}_{\varphi}$; equivalently, there is $\xi=\xi(\om)\in\Ltwo{\G}$
such that $\om(x^{*})=\left\langle \xi,\Lambda(x)\right\rangle $
for every $x\in\mathcal{N}_{\varphi}$. Then $\hat{\varphi}$ is the
unique n.s.f.~weight on $\Linfty{\hat{\G}}$ whose GNS construction
$(\Ltwo{\G},\hat{\Lambda})$ satisfies $\hat{\Lambda}(\l(\om))=\xi(\om)$
for all $\om\in\mathcal{I}$ and that $\l(\mathcal{I})$ is a $*$-ultrastrong--norm
core for $\hat{\Lambda}$.

A fundamental object for $\G$ is its \emph{antipode} $S$, which
is a $*$-ultrastrongly closed, densely defined, generally unbounded
linear operator on $\Linfty{\G}$. It has the ``polar decomposition''
$S=R\circ\tau_{-i/2}$, where $R$ stands for the \emph{unitary antipode}
and $\left(\tau_{t}\right)_{t\in\R}$ for the \emph{scaling group}.
We will not discuss here the definitions of these maps. The subspace
\[
\LoneStar{\G}:=\left\{ \om\in\LoneStar{\G}:(\exists\rho\in\Lone{\G}\,\forall x\in D(S))\quad\rho(x)=\overline{\om}(S(x))\right\} 
\]
is a dense subalgebra of $\Lone{\G}$. For $\om\in\LoneStar{\G}$,
let $\om^{*}$ be the unique element $\rho\in\Lone{\G}$ such that
$\rho(x)=\overline{\om}(S(x))$ for each $x\in D(S)$. Then $\om\mapsto\om^{*}$
is an involution on $\LoneStar{\G}$, and $\l|_{\LoneStar{\G}}$ is
a $*$-homomorphism. Moreover, $\LoneStar{\G}$ is an involutive Banach
algebra when equipped with the new norm $\left\Vert \om\right\Vert _{*}:=\max(\left\Vert \om\right\Vert ,\left\Vert \om^{*}\right\Vert )$.

A useful construction is the opposite LCQG $\G^{\mathrm{op}}$ \citep[Section 4]{Kustermans_Vaes__LCQG_von_Neumann},
which has $\Linfty{\G^{\mathrm{op}}}:=\Linfty{\G}$ and co-multiplication
given by $\Delta^{\mathrm{op}}:=\sigma\circ\Delta$.

The \emph{universal} setting of $\G$ was defined by Kustermans \citep{Kustermans__LCQG_universal}
as follows. Let $\CzU{\G}$ be the enveloping $C^{*}$-algebra of
$\LoneStar{\hat{\G}}$. The canonical embedding of $\LoneStar{\hat{\G}}$
in $\CzU{\G}$ is denoted by $\hat{\l}_{\mathrm{u}}$. By universality,
there exists a surjective $*$-homomorphism $\pi_{\mathrm{u}}:\CzU{\G}\to\Cz{\G}$
satisfying $\pi_{\mathrm{u}}(\hat{\l}_{\mathrm{u}}(\om))=\hat{\l}(\om)$
for every $\om\in\LoneStar{\hat{\G}}$. There exists a co-multiplication
$\Delta_{\mathrm{u}}:\CzU{\G}\to M(\CzU{\G}\tensormin\CzU{\G})$ satisfying
$(\pi_{\mathrm{u}}\tensor\pi_{\mathrm{u}})\Delta_{\mathrm{u}}=\Delta\pi_{\mathrm{u}}$,
inducing a convolution in $\CzU{\G}^{*}$, making it an involutive
Banach algebra. Using the isometry $\pi_{\mathrm{u}}^{*}:\Cz{\G}^{*}\to\CzU{\G}^{*}$,
one can see $\Cz{\G}^{*}$ as a subset of $\CzU{\G}^{*}$, which is
a (closed, two-sided) ideal. Furthermore, $\Lone{\G}$ is also a (closed,
two-sided) ideal in $\CzU{\G}^{*}$ \citep[Proposition 8.3]{Daws__mult_self_ind_dual_B_alg}.

The left regular co-representation of $\G$ has a universal version.
It is a unitary $\WW\in M(\CzU{\G}\tensormin\CzU{\hat{\G}})$ satisfying
$(\Delta_{\mathrm{u}}\tensor\i)(\WW)=\WW_{13}\WW_{23}$ and $(\pi_{\mathrm{u}}\tensor\hat{\pi}_{\mathrm{u}})(\WW)=W$.
Its dual object is $\hat{\WW}=\sigma(\WW^{*})$. Letting $\Ww:=(\i\tensor\hat{\pi}_{\mathrm{u}})(\WW)$
and $\wW:=(\pi_{\mathrm{u}}\tensor\i)(\WW)$, we have $\Ww\in M(\CzU{\G}\tensormin\Cz{\hat{\G}})$,
$\wW\in M(\Cz{\G}\tensormin\CzU{\hat{\G}})$ and $(\i\tensor\pi_{\mathrm{u}})\Delta_{\mathrm{u}}(x)=\Ww^{*}(\one\tensor\pi_{\mathrm{u}}(x))\Ww$
for every $x\in\CzU{\G}$. Moreover, representing $\CzU{\G}$ faithfully
on a Hilbert space $\H_{\mathrm{u}}$ and viewing the operator $\Ww\in M(\CzU{\G}\tensormin\Cz{\hat{\G}})$
as an element of $B(\H_{\mathrm{u}}\tensor\Ltwo{\G})$, we have $\Ww\in M(\CzU{\G}\tensormin\K(\Ltwo{\G}))$.
Also $\l_{\mathrm{u}}(\om)=(\om\tensor\i)(\wW)$ for every $\om\in\LoneStar{\G}$,
and the map $\l^{\mathrm{u}}:\CzU{\G}^{*}\to M(\Cz{\hat{\G}})$, $\om\mapsto(\om\tensor\i)(\Ww)$
for $\om\in\CzU{\G}^{*}$, is a $*$-homomorphism. 

The universality property of $\CzU{\G}$ implies the existence of
the \emph{co-unit}, which is the unique $*$-homomorphism $\epsilon\in\CzU{\G}_{+}^{*}$
such that $(\epsilon\tensor\i)\circ\Delta_{\mathrm{u}}=\i=(\i\tensor\epsilon)\circ\Delta_{\mathrm{u}}$.
It satisfies $(\epsilon\tensor\i)(\WW)=\one_{M(\CzU{\hat{\G}})}$.

For a Banach algebra $A$, the canonical module action of $A$ on
its dual $A^{*}$ is denoted by juxtaposition, that is,
\[
(\mu a)(b)=\mu(ab)\quad\text{and}\quad(a\mu)(b)=\mu(ba)\qquad(\forall\mu\in A^{*},a,b\in A).
\]
This notation will be used for the actions of $\Linfty{\G}$, $\Cz{\G}$
and $\CzU{\G}$ on their duals.

The canonical module actions of $\Lone{\G}$ on $\Linfty{\G}$ will
be denoted by `$\cdot$', so we have 
\[
\om\cdot a=(\i\tensor\om)\Delta(a)\quad\text{and}\quad a\cdot\om=(\om\tensor\i)\Delta(a)\qquad(\forall\om\in\Lone{\G},a\in\Linfty{\G}).
\]
Each of $\left\{ \om\cdot a:\om\in\Lone{\G},a\in\Cz{\G}\right\} $
and $\left\{ a\cdot\om:\om\in\Lone{\G},a\in\Cz{\G}\right\} $ spans
a norm dense subset of $\Cz{\G}$. 

More generally, every $\mu\in\CzU{\G}^{*}$ acts on $\Linfty{\G}$
as follows: for $a\in\Linfty{\G}$, $\mu\cdot a$ and $a\cdot\mu$
are defined to be the unique elements of $\Linfty{\G}$ satisfying
\[
\om(\mu\cdot a)=(\om*\mu)(a),\quad\om(a\cdot\mu)=(\mu*\om)(a)\qquad(\forall\om\in\Lone{\G}).
\]
Note that if $\mu_{1},\mu_{2}\in\CzU{\G}^{*}$ and $a\in\Linfty{\G}$,
then 
\[
\om[\mu_{1}\cdot(\mu_{2}\cdot a)]=(\om*\mu_{1})(\mu_{2}\cdot a)=(\om*\mu_{1}*\mu_{2})(a)=\om[(\mu_{1}*\mu_{2})\cdot a],
\]
thus $\mu_{1}\cdot(\mu_{2}\cdot a)=(\mu_{1}*\mu_{2})\cdot a$. Similarly,
$(a\cdot\mu_{1})\cdot\mu_{2}=a\cdot(\mu_{1}*\mu_{2})$. 
\begin{lem}
\label{lem:C_0__C_0_u_star}If $a\in\Cz{\G}$ and $\mu\in\CzU{\G}^{*}$,
then $\mu\cdot a,a\cdot\mu\in\Cz{\G}$. \end{lem}
\begin{proof}
Fix $\mu\in\CzU{\G}^{*}$. If $\om\in\Lone{\G}$ and $b\in\Cz{\G}$,
then $\mu\cdot(\om\cdot b)=(\mu*\om)\cdot b\in\Cz{\G}$ as $\mu*\om\in\Lone{\G}$.
By density, $\mu\cdot a\in\Cz{\G}$ for all $a\in\Cz{\G}$. The proof
for $a\cdot\mu$ is similar.
\end{proof}

\subsection{Types of LCQGs}

Compact quantum groups were introduced by Woronowicz in \citep{Woronowicz__symetries_quantiques},
and discrete quantum groups by Effros and Ruan \citep{Effros_Ruan__discrete_QGs}
and by Van Daele \citep{Van_Daele__discrete_CQs}. We will not present
their original definitions, but define them through the Kustermans--Vaes
axiomatization. Complete proofs of the equivalence of various characterizations
of compact and discrete quantum groups can be found in \citep{Runde__charac_compact_discr_QG}.

A LCQG $\G$ is \emph{compact} if its left Haar weight $\varphi$
is finite. This is equivalent to $\Cz{\G}$ being unital. In this
case, we denote $\Cz{\G}$ by $C(\G)$. Moreover, the right Haar weight
$\psi$ is also finite, and assuming, as customary, that both $\varphi$
and $\psi$ are states, they are equal. 

A LCQG $\G$ is \emph{discrete} if it is the dual of a compact quantum
group. This is equivalent to $(\Lone{\G},*)$ admitting a unit $\epsilon$.
In this case, we denote $\Cz{\G},\Linfty{\G}$ by $c_{0}(\G),\ell^{\infty}(\G)$,
respectively, and have 
\[
c_{0}(\G)\cong c_{0}-\bigoplus_{\a\in\mathrm{Irred}(\hat{\G})}M_{n(\a)}\quad\text{and}\quad\ell^{\infty}(\G)\cong\ell^{\infty}-\bigoplus_{\a\in\mathrm{Irred}(\hat{\G})}M_{n(\a)},
\]
where $\mathrm{Irred}(\G)$ is the set of equivalence classes of (necessarily
finite-dimensional) irreducible unitary co-representations of $\hat{\G}$,
and for every $\a\in\mathrm{Irred}(\hat{\G})$, $n(\a)\in\N$ denotes
the dimension of the representation. Particularly, the summand corresponding
to the trivial co-representation of $\hat{\G}$ gives a central minimal
projection $p$ in $\ell^{\infty}(\G)$, satisfying $ap=\epsilon(a)p=pa$
for every $a\in\ell^{\infty}(\G)$. 

A LCQG $\G$ is called \emph{co-amenable} (see B{\'e}dos and Tuset
\citep{Bedos_Tuset_2003} or Desmedt, Quaegebeur and Vaes \citep{Desmedt_Quaegebeur_Vaes},
who use a different terminology) if $\Lone{\G}$ admits a bounded
approximate identity. This is equivalent to the Banach algebra $(\Cz{\G}^{*},*)$
having a unit \citep[Theorem 3.1]{Bedos_Tuset_2003}, which is called
the \emph{co-unit} of $\G$ and denoted by $\epsilon$. It is also
equivalent to the surjection $\pi_{\mathrm{u}}:\CzU{\G}\to\Cz{\G}$
being an isomorphism, in which case we simply identify $\CzU{\G}$
with $\Cz{\G}$.

Every locally compact group $G$ is co-amenable as a (commutative)
quantum group, while its co-commutative dual $\hat{G}$ is co-amenable
if and only if $G$ is amenable as a group. Discrete quantum groups
are trivially co-amenable.

\subsection{Positive-definite functions over LCQGs}

Let $\G$ be a LCQG. In \citep{Daws__CPM_LCQGs_2012,Daws_Salmi__CPD_LCQGs_2013},
Daws and Salmi introduced four notions of positive definiteness for
elements of $\Linfty{\G}$. Here we will need only two of them, namely
(1) and (2) of \citep{Daws_Salmi__CPD_LCQGs_2013}. Note that we use
different notation: $\overline{\om},\om^{*}$ are denoted by $\om^{*},\om^{\sharp}$
in \citep{Daws__CPM_LCQGs_2012,Daws_Salmi__CPD_LCQGs_2013}.
\begin{defn}
Let $\G$ be a LCQG. 
\begin{enumerate}
\item \label{enu:def_pos_def_func}A \emph{positive-definite function} is
$x\in\Linfty{\G}$ satisfying $(\om^{*}*\om)(x^{*})\geq0$ for every
$\om\in\LoneStar{\G}$. 
\item \label{enu:def_FS_transform_pos_mes}A \emph{Fourier--Stieltjes transform
of a positive measure} is an element $x$ of the form $(\i\tensor\hat{\mu})(\wW^{*})=\mathbf{\hat{\l}^{\mathrm{u}}}(\hat{\mu})$
for some $\hat{\mu}\in\CzU{\hat{\G}}_{+}^{*}$. Note that $x\in M(\Cz{\G})$
in this case.
\end{enumerate}
\end{defn}
\begin{thm}[{\citep[Lemma 1 and Theorem 15]{Daws_Salmi__CPD_LCQGs_2013}}]
\label{thm:pos_def}For $x\in\Linfty{\G}$, we have \prettyref{enu:def_FS_transform_pos_mes}$\implies$\prettyref{enu:def_pos_def_func},
and the converse holds when $\G$ is co-amenable.
\end{thm}
For co-amenable $\G$, we will therefore just use the adjective ``positive
definite'' for these elements.
\begin{rem}
\label{rem:CPD_norm}Let $\G$ be a co-amenable LCQG with co-unit
$\epsilon\in\Cz{\G}^{*}$. Write $\epsilon$ also for its strictly
continuous extension to $M(\Cz{\G})$. If $x\in\Linfty{\G}$ is positive
definite, then $\left\Vert x\right\Vert =\epsilon(x)$, for writing
$x=(\i\tensor\hat{\mu})(\wW^{*})$ with $\hat{\mu}\in\CzU{\hat{\G}}_{+}^{*}$,
we have 
\[
\left\Vert x\right\Vert \geq\epsilon(x)=\epsilon((\i\tensor\hat{\mu})(\wW^{*}))=\hat{\mu}((\epsilon\tensor\i)(\wW^{*}))=\hat{\mu}(\one)=\left\Vert \hat{\mu}\right\Vert \geq\left\Vert x\right\Vert 
\]
(see \citep[Corollary 2.2]{Bedos_Murphy_Tuset__am_and_coam} and \citep[Theorem 3.1]{Bedos_Tuset_2003}).
\end{rem}

\section{\label{sec:square_integrable_positive_def}Square-integrable positive-definite
functions over locally compact quantum groups}

This section is dedicated to proving a generalization of Godement's
theorem on square-integrable positive-definite functions. It can be
established directly along the lines of \citep[Section 13.8]{Dixmier__C_star_English},
but we feel that it is more correct to do it through the generalization
of this result to left Hilbert algebras given by Phillips \citep{Phillips__pos_int_elem}.
We start with some background. Let $\Aa$ be a full (that is, achieved)
left Hilbert algebra \citep{Takesaki__Tomita_thy,Takesaki__book_vol_2}
and $\H$ be the completion of $\Aa$. We denote by $\pi(\xi)$ (resp.
$\pi'(\xi)$) the operator corresponding to a left-bounded (resp.
right-bounded) vector $\xi\in\H$.
\begin{defn}[Perdrizet \citep{Perdrizet__elements_positifs}, Haagerup \citep{Haagerup__standard_form}]
\label{def:P_flat}Let $\mathcal{P}^{\flat}:=\left\{ \eta\in\H:\left\langle \eta,\xi^{\sharp}\xi\right\rangle \geq0\text{ for every \ensuremath{\xi\in\Aa}}\right\} $. 
\end{defn}
This set is evidently a cone in $\H$.
\begin{rem}
\label{rem:pos_left_bounded}Let $\eta\in\H$. \citep[Theorem VI.1.26 (ii)]{Takesaki__book_vol_2}
implies that $\eta\in\mathcal{P}^{\flat}$ if and only if $\left\langle \eta,\pi(\xi)^{*}\xi\right\rangle \geq0$
for every left-bounded vector $\xi\in\H$.\end{rem}
\begin{defn}[\citep{Phillips__pos_int_elem}]
\label{def:integrable_and_root}Let $\eta\in\mathcal{P}^{\flat}$.
\begin{enumerate}
\item Say that $\eta$ is \emph{integrable} if $\sup\left\{ \left\langle \eta,\xi\right\rangle :\xi\text{ is a selfadjoint idempotent in }\Aa\right\} <\infty$. 
\item Say that $\z\in\mathcal{P}^{\flat}$ is a \emph{square root} of $\eta$
if $\left\langle \xi,\eta\right\rangle =\left\langle \pi(\xi)\z,\z\right\rangle $
for every $\xi\in\Aa$.
\end{enumerate}
\end{defn}
We denote the set of all integrable elements of $\mathcal{P}^{\flat}$
by $\mathcal{P}_{\mathrm{int}}^{\flat}$.
\begin{thm}[{\citep[Theorem 1.10]{Phillips__pos_int_elem}}]
\label{thm:Phillips_root}Let $\eta\in\mathcal{P}^{\flat}$. Then
$\eta$ is integrable if and only if it has a square root $\z\in\mathcal{P}^{\flat}$.
If $\eta\in\Aa'$, then also $\zeta\in\Aa'$, and $\z\z=\eta$.
\end{thm}
Moreover, the span of $\mathcal{P}_{\mathrm{int}}^{\flat}$ can be
endowed with a natural norm making it isometrically isomorphic to
a dense subspace of the predual of the (left) von Neumann algebra
$\mathcal{R}_{\ell}(\Aa)$ of $\Aa$ \citep[Theorem 2.9]{Phillips__pos_int_elem}.
In particular, $\eta\in\mathcal{P}_{\mathrm{int}}^{\flat}$ with square
root $\z\in\mathcal{P}^{\flat}$ induces the element $\om_{\z}|_{\mathcal{R}_{\ell}(\Aa)}$
of $\mathcal{R}_{\ell}(\Aa)_{*}$. 

Let $\G$ be a LCQG, and set $\mathcal{J}:=\mathcal{I}\cap\LoneStar{\G}$.
\begin{lem}
\label{lem:J_and_S}Let $x,y\in\Linfty{\G}$. If $(\om_{1}^{*}*\om_{2})^{*}(y)=(\overline{\om_{1}^{*}}*\overline{\om_{2}})(x)$
for every $\om_{1},\om_{2}\in\mathcal{J}$, then $y\in D(S)$ and
$S(y)=x$.\end{lem}
\begin{proof}
The assertion follows by repeating the argument of \citep[proof of Lemma 5]{Daws_Salmi__CPD_LCQGs_2013}
with $\LoneStar{\G}$ being replaced by $\mathcal{J}$. This is possible
as $\mathcal{I}$, and hence $\mathcal{J}$, are invariant under the
scaling group adjoint $\left(\tau_{t}^{*}\right)_{t\in\R}$, and $\mathcal{J},\mathcal{J}^{*}$
are norm dense in $\Lone{\G}$ \citep[Lemma 2.5 and its proof]{Kustermans_Vaes__LCQG_von_Neumann}.
\end{proof}
We need a slight strengthening of \citep[Theorem 6]{Daws_Salmi__CPD_LCQGs_2013}
and part of \citep[Proposition 2.6]{Kustermans_Vaes__LCQG_von_Neumann}.
\begin{lem}
\label{lem:J_cap_J_star_dense}The set $\left\{ \om_{1}^{*}*\om_{2}:\om_{1},\om_{2}\in\mathcal{J}\right\} $
is total in $(\LoneStar{\G},\left\Vert \cdot\right\Vert _{*})$. Thus
the subspace $\mathcal{J}\cap\mathcal{J}^{*}$ is dense in $(\LoneStar{\G},\left\Vert \cdot\right\Vert _{*})$.\end{lem}
\begin{proof}
Since $\mathcal{I}$ is a left ideal \citep[Lemma 4.8]{Van_Daele__LCQGs},
$\left\{ \om_{1}^{*}*\om_{2}:\om_{1},\om_{2}\in\mathcal{J}\right\} $
is contained in $\mathcal{J}\cap\mathcal{J}^{*}$. Adapting the argument
of \citep[proof of Theorem 6]{Daws_Salmi__CPD_LCQGs_2013}, if $\left\{ \om_{1}^{*}*\om_{2}:\om_{1},\om_{2}\in\mathcal{J}\right\} $
were not total in $(\LoneStar{\G},\left\Vert \cdot\right\Vert _{*})$,
then there would be $x,y\in\Linfty{\G}$ such that 
\[
0=(\om_{1}^{*}*\om_{2})(x)+\overline{(\om_{1}^{*}*\om_{2})^{*}(y)},
\]
that is, $(\om_{1}^{*}*\om_{2})^{*}(y)=(\overline{\om_{1}^{*}}*\overline{\om_{2}})(-x^{*})$,
for every $\om_{1},\om_{2}\in\mathcal{J}$. \prettyref{lem:J_and_S}
gives that $y\in D(S)$ and $S(y)=-x^{*}$, and hence the element
of $(\LoneStar{\G},\left\Vert \cdot\right\Vert _{*})^{*}$ corresponding
to $(x,\overline{y})$ is zero.
\end{proof}
Considering the full left Hilbert algebra $\Aa_{\hat{\varphi}}$ associated
with the left-invariant weight $\hat{\varphi}$ of $\hat{\G}$, we
let $\mathcal{P}_{\hat{\varphi}}^{\flat}$ stand for the corresponding
cone.
\begin{lem}
\label{lem:defs_of_PD}Let $x\in\Linfty{\G}$. If $x\in\mathcal{N}_{\varphi}$,
then $x$ is a positive-definite function if and only if $\Lambda(x)\in\mathcal{P}_{\hat{\varphi}}^{\flat}$. \end{lem}
\begin{proof}
By definition, $x$ is positive definite if and only if $(\om^{*}*\om)(x^{*})\geq0$
for every $\om\in\LoneStar{\G}$. From \prettyref{lem:J_cap_J_star_dense},
it suffices to check this for $\om\in\mathcal{J}$. But if $\om\in\mathcal{J}$,
then also $\om^{*}*\om\in\mathcal{J}$ and for $\hat{y}:=\l(\om)$
we have $\hat{y}^{*}\hat{y}=\l(\om^{*}*\om)$ and 
\[
(\om^{*}*\om)(x^{*})=\langle\hat{\Lambda}(\l(\om^{*}*\om)),\Lambda(x)\rangle=\langle\hat{y}^{*}\hat{\Lambda}(\hat{y}),\Lambda(x)\rangle.
\]
By \prettyref{rem:pos_left_bounded}, $\Lambda(x)\in\mathcal{P}_{\hat{\varphi}}^{\flat}$
if and only if $\langle\Lambda(x),\hat{y}^{*}\hat{\Lambda}(\hat{y})\rangle\geq0$
for every $\hat{y}\in\mathcal{N}_{\hat{\varphi}}$. Using \citep[Lemma 2.5]{Kustermans_Vaes__LCQG_von_Neumann},
that is equivalent to $\langle\Lambda(x),\hat{y}^{*}\hat{\Lambda}(\hat{y})\rangle\geq0$
for every $\hat{y}\in\l(\mathcal{J})$. This completes the proof.\end{proof}
\begin{prop}
\label{prop:approx_iden_special_features}Let $\G$ be a co-amenable
LCQG. There exists a contractive approximate identity for $(\LoneStar{\G},\left\Vert \cdot\right\Vert _{*})$
in $\mathcal{J}\cap\mathcal{J}^{*}$.\end{prop}
\begin{proof}
By \citep[Theorem 13]{Daws_Salmi__CPD_LCQGs_2013}, $(\LoneStar{\G},\left\Vert \cdot\right\Vert _{*})$
has a contractive approximate identity. Combining this with \prettyref{lem:J_cap_J_star_dense},
the assertion is proved.
\end{proof}
The following result generalizes \citep[Theorem 1.6]{Phillips__pos_int_elem},
saying that if $G$ is a locally compact group and $f\in\Ltwo G$
is positive definite and essentially bounded on a neighborhood of
the identity, then it belongs to $A(G)$.
\begin{cor}
\label{cor:pos_def_L2_imply_integrable}Let $\G$ be a co-amenable
LCQG. If $x\in\mathcal{N}_{\varphi}$ and $x$ is positive definite,
then $\Lambda(x)$ is integrable with respect to $\Aa_{\hat{\varphi}}$
(see \prettyref{def:integrable_and_root}). \end{cor}
\begin{proof}
Let $\left(\epsilon_{i}\right)$ be a contractive approximate identity
for $(\LoneStar{\G},\left\Vert \cdot\right\Vert _{*})$ in $\mathcal{J}\cap\mathcal{J}^{*}$.
Then letting $\xi_{i}:=\hat{\Lambda}(\l(\epsilon_{i}))$, we get a
net $\left(\xi_{i}\right)$ in the left Hilbert algebra $\Aa_{\hat{\varphi}}$.
Since $x\in\mathcal{N}_{\varphi}$, we have for every $i$, 
\[
\langle\Lambda(x),\xi_{i}^{\sharp}\xi_{i}\rangle=\langle\Lambda(x),\hat{\Lambda}(\l(\epsilon_{i}^{*}*\epsilon_{i}))\rangle=\overline{\langle\hat{\Lambda}(\l(\epsilon_{i}^{*}*\epsilon_{i})),\Lambda(x)\rangle}=\overline{(\epsilon_{i}^{*}*\epsilon_{i})(x^{*})},
\]
and so $\langle\Lambda(x),\xi_{i}^{\sharp}\xi_{i}\rangle\leq\left\Vert \epsilon_{i}^{*}\right\Vert \left\Vert \epsilon_{i}\right\Vert \left\Vert x\right\Vert \leq\left\Vert x\right\Vert $.
Since $\Lambda(x)\in\mathcal{P}_{\hat{\varphi}}^{\flat}$ by \prettyref{lem:defs_of_PD}
and $\left(\l(\epsilon_{i})\right)$ converges strongly to $\one$
(for $\LoneStar{\G}$ is dense in $\Lone{\G}$), \citep[Proposition 1.5]{Phillips__pos_int_elem}
applies, and yields that $\Lambda(x)$ is integrable with respect
to $\Aa_{\hat{\varphi}}$.
\end{proof}
We now prove the main result of this section, generalizing a theorem
of Godement \citep[Th\'{e}or\`{e}me 17]{Godement__positive_def_func}.
\begin{thm}
\label{thm:Godement}Let $\G$ be a co-amenable LCQG. If $x\in\mathcal{N}_{\varphi}$
and $x$ is positive definite, then $\Lambda(x)$ has a square root
in $\mathcal{P}_{\hat{\varphi}}^{\flat}$ (\prettyref{def:integrable_and_root});
equivalently, there exists $\z\in\mathcal{P}_{\hat{\varphi}}^{\flat}$
such that $x=\hat{\l}(\hat{\om}_{\z})$. If, additionally, $\Lambda(x)\in\Aa_{\hat{\varphi}}'$,
then also $\z\in\Aa_{\hat{\varphi}}'$, in which case $\Lambda(x)=\hat{\pi}'(\z)\z$.
That is, if $\hat{w}\in\mathcal{N}_{\hat{\varphi}}$ is positive and
$\hat{J}\hat{\Lambda}(\hat{w})\in\Lambda(\mathcal{N}_{\varphi})$,
then the (positive) square root of $\hat{w}$ also belongs to $\mathcal{N}_{\hat{\varphi}}$.\end{thm}
\begin{proof}
The first part of the first assertion, as well as the second assertion,
follow from \prettyref{thm:Phillips_root} by using \prettyref{cor:pos_def_L2_imply_integrable}.
For the part after ``equivalently'', $\Lambda(x)$ having a square
root in $\mathcal{P}_{\hat{\varphi}}^{\flat}$ means, by definition,
that there exists $\z\in\mathcal{P}_{\hat{\varphi}}^{\flat}$ such
that $\langle\hat{\Lambda}(\hat{y}),\Lambda(x)\rangle=\hat{\om}_{\z}(\hat{y})$
for every $\hat{y}\in\mathcal{N}_{\hat{\varphi}}\cap\mathcal{N}_{\hat{\varphi}}^{*}$,
thus for every $\hat{y}\in\mathcal{N}_{\hat{\varphi}}$ \citep[Theorem VI.1.26 (ii)]{Takesaki__book_vol_2}.
In particular, for every $\om\in\mathcal{I}$, 
\[
\om(x^{*})=\langle\hat{\Lambda}(\l(\om)),\Lambda(x)\rangle=\hat{\om}_{\z}(\l(\om))=\om[(\i\tensor\hat{\om}_{\z})(W)].
\]
The density of $\mathcal{I}$ in $\Lone{\G}$ entails that $x=(\i\tensor\hat{\om}_{\z})(W)^{*}=(\i\tensor\hat{\om}_{\z})(W^{*})=\hat{\l}(\hat{\om}_{\z})$.
The converse is proved similarly.

For the last sentence, note that $\mathcal{P}_{\hat{\varphi}}^{\flat}\cap\Aa_{\hat{\varphi}}'=\{\hat{J}\hat{\Lambda}(\hat{w}):\hat{w}\in\mathcal{N}_{\hat{\varphi}}\text{ and }\hat{w}\geq0\}$
(see, e.g., the right version of \citep[Proposition 2.5]{Perdrizet__elements_positifs}).
If an element there is in $\Lambda(\mathcal{N}_{\varphi})$, then
its square root in $\mathcal{P}_{\hat{\varphi}}^{\flat}$ has the
form $\hat{J}\hat{\Lambda}(\hat{z})$ for $\hat{z}\in\mathcal{N}_{\hat{\varphi}}$
with $\hat{z}\geq0$, and the equality $\hat{J}\hat{\Lambda}(\hat{w})=\hat{J}\hat{\Lambda}(\hat{z}^{2})$
implies that $\hat{w}=\hat{z}^{2}$.\end{proof}
\begin{rem}
In the situation of \prettyref{thm:Godement} we have $\left\Vert x\right\Vert =\left\Vert \hat{\om}_{\z}\right\Vert =\left\Vert \z\right\Vert ^{2}$
by \prettyref{rem:CPD_norm}.
\end{rem}

\section{\label{sec:convolution_L_p}Convolution in $L^{p}(\protect\G)$}

This section contains the preliminaries on non-commutative $L^{p}$-spaces
of LCQGs needed in the next section. The theory of non-commutative
$L^{p}$-spaces of von Neumann algebras was developed in three approaches,
which turned out to be equivalent: the ``abstract'' one by Haagerup
\citep{Haagerup__L_p_spaces}, the ``spatial'' one by Connes and
Hilsum \citep{Hilsum__L_p_spaces}, and the one using interpolation
theory, whose final form is by Izumi \citep{Izumi__noncomm_Lp} (see
also Terp \citep{Terp__L_p_spaces,Terp__interpolation}). Here we
rely on the work of Caspers \citep{Caspers__Lp_Fourier_LCQG}, who
introduced and studied non-commutative $L^{p}$-spaces of LCQGs based
on Izumi's approach with interpolation parameter $\a=-\frac{1}{2}$.
This has two clear virtues. The first, which is intrinsic in interpolation
theory, is the fact that all non-commutative $L^{p}$-spaces are realized,
as vector spaces, as subspaces of a larger space, allowing the consideration
of intersections of them. Caspers proved that when $\a=-\frac{1}{2}$,
some of these intersections take a particularly natural form. The
second is simplicity: the statement (but not proofs!) of the construction's
basic ingredients does not require modular theory.

We bring now a succinct account of the theory. A pair of Banach spaces
$(A_{0},A_{1})$ is called \emph{compatible }in the sense of interpolation
theory (see Bergh and L{\"o}fstr{\"o}m \citep[Section 2.3]{Bergh_Lfstrom__book})
if they are continuously embedded in a Hausdorff topological vector
space. For $0<\theta<1$, the Calder\'{o}n \emph{complex interpolation
method} \citep[Chapter 4]{Bergh_Lfstrom__book} gives the interpolation
Banach space $C_{\theta}(A_{0},A_{1})$. As a vector space it satisfies
$A_{0}\cap A_{1}\subseteq C_{\theta}(A_{0},A_{1})\subseteq A_{0}+A_{1}$,
and these inclusions are contractive when $A_{0}\cap A_{1}$ and $A_{0}+A_{1}$
are given the norms $\left\Vert a\right\Vert _{A_{0}\cap A_{1}}:=\max(\left\Vert a\right\Vert _{A_{0}},\left\Vert a\right\Vert _{A_{1}})$,
$a\in A_{0}\cap A_{1}$, and $\left\Vert a\right\Vert _{A_{0}+A_{1}}:=\inf\left\{ \left\Vert a_{0}\right\Vert _{A_{0}}+\left\Vert a_{1}\right\Vert _{A_{1}}:a_{0}\in A_{0},a_{1}\in A_{1},a=a_{0}+a_{1}\right\} $,
$a\in A_{0}+A_{1}$. Moreover, $A_{0}\cap A_{1}$ is dense in $C_{\theta}(A_{0},A_{1})$
\citep[Theorem 4.2.2]{Bergh_Lfstrom__book}. The functor $C_{\theta}$
is an \emph{exact interpolation functor of exponent $\theta$} in
the following sense. Given another compatible pair $(B_{0},B_{1})$,
two bounded maps $T_{i}:A_{i}\to B_{i}$, $i=0,1$, are called \emph{compatible}
if they agree on $A_{0}\cap A_{1}$. Then the induced linear map $T:A_{0}+A_{1}\to B_{0}+B_{1}$
satisfies $TC_{\theta}(A_{0},A_{1})\subseteq C_{\theta}(B_{0},B_{1})$,
and the restriction $T:C_{\theta}(A_{0},A_{1})\to C_{\theta}(B_{0},B_{1})$
has norm at most $\left\Vert T_{0}\right\Vert ^{1-\theta}\left\Vert T_{1}\right\Vert ^{\theta}$.

Let $M$ be a von Neumann algebra, and let $\varphi$ be an n.s.f.~weight
on $M$. Define 
\[
L:=\left\{ x\in\mathcal{N}_{\varphi}:(\exists\presb x{\varphi}\in M_{*}\forall y\in\mathcal{N}_{\varphi})\quad\presb x{\varphi}(y^{*})=\varphi(y^{*}x)\right\} ,
\]
\[
R:=\left\{ x\in\mathcal{N}_{\varphi}^{*}:(\exists\varphi_{x}\in M_{*}\forall y\in\mathcal{N}_{\varphi})\quad\varphi_{x}(y)=\varphi(xy)\right\} .
\]
The spaces $L,R$ are precisely $L_{(-1/2)},L_{(1/2)}$ in Izumi's
notation \citep[Proposition 2.14]{Caspers__Lp_Fourier_LCQG}. Endow
$L,R$ with norms by putting $\left\Vert x\right\Vert _{L}:=\max(\left\Vert x\right\Vert _{M},\left\Vert \presb x{\varphi}\right\Vert _{M_{*}})$
for $x\in L$ and $\left\Vert x\right\Vert _{R}:=\max(\left\Vert x\right\Vert _{M},\left\Vert \varphi_{x}\right\Vert _{M_{*}})$
for $x\in R$. Define linear mappings $l^{1}:L\to M_{*}$, $l^{\infty}:L\to M$,
$r^{1}:R\to M_{*}$ and $r^{\infty}:R\to M$ by $l^{1}(x):=\presb x{\varphi}$
and $l^{\infty}(x):=x$ for $x\in L$, and similarly $r^{1}(x):=\varphi_{x}$
and $r^{\infty}(x):=x$ for $x\in R$. These maps are contractive
and injective. Furthermore, the adjoints $(l^{1})^{*}:M\to L^{*}$,
$(l^{\infty})^{*}:M_{*}\to L^{*}$, $(r^{1})^{*}:M\to R^{*}$ and
$(r^{\infty})^{*}:M_{*}\to R^{*}$ are also injective (in the second
and the fourth we restricted the usual adjoint from $M^{*}$ to $M_{*}$).
By \citep[Theorem 2.5]{Izumi__noncomm_Lp}, the diagram on the left-hand
side is commutative:
\begin{align}
 & \xymatrix{ & M\ar@{^{(}->}[rd]^{(r^{1})^{*}}\\
L\ar@{^{(}->}[ru]^{l^{\infty}}\ar@{^{(}->}[rd]_{l^{1}} &  & R^{*}\\
 & M_{*}\ar@{^{(}->}[ru]_{(r^{\infty})^{*}}
}
 &  & \xymatrix{ & M\ar@{^{(}->}[rd]^{(r^{1})^{*}}\\
L\ar@{^{(}->}[ru]^{l^{\infty}}\ar@{^{(}->}[rd]_{l^{1}}\ar@{^{(}->}[r]^{l^{p}} & L^{p}(M)_{\mathrm{left}}\ar@{^{(}->}[r] & R^{*}\\
 & M_{*}\ar@{^{(}->}[ru]_{(r^{\infty})^{*}}
}
\label{eq:Izumi_diagram}
\end{align}
In addition, by \citep[Corollary 2.13]{Izumi__noncomm_Lp}, 
\begin{equation}
((r^{1})^{*}\circ l^{\infty})(L)=(r^{1})^{*}(M)\cap(r^{\infty})^{*}(M_{*})=((r^{\infty})^{*}\circ l^{1})(L),\label{eq:Izumi_interpolation_intersect}
\end{equation}
allowing to regard $L$ as the ``intersection of $M$ and $M_{*}$
in $R^{*}$''.

Viewing $M,M_{*}$ as embedded, as vector spaces, in $R^{*}$ via
$(r^{1})^{*},(r^{\infty})^{*}$, the pair $(M,M_{*})$ is thus compatible.
For $1<p<\infty$, we define $(L^{p}(M)_{\mathrm{left}},\left\Vert \cdot\right\Vert _{p})$
to be the interpolation Banach space $C_{1/p}(M,M_{*})$. As above,
we have $(r^{1})^{*}(M)\cap(r^{\infty})^{*}(M_{*})\subseteq L^{p}(M)_{\mathrm{left}}\subseteq(r^{1})^{*}(M)+(r^{\infty})^{*}(M_{*})$
(all inside $R^{*}$) with contractive inclusions and $(r^{1})^{*}(M)\cap(r^{\infty})^{*}(M_{*})$
is dense in $(L^{p}(M)_{\mathrm{left}},\left\Vert \cdot\right\Vert _{p})$.
From \prettyref{eq:Izumi_interpolation_intersect} we get a contractive
injection $l^{p}:L\to L^{p}(M)_{\mathrm{left}}$ with dense range,
and the diagram on the right-hand side of \prettyref{eq:Izumi_diagram}
is commutative. 

Denote by $(\H,\Lambda)$ the GNS construction for $(M,\varphi)$.
The map $l^{2}(x)\mapsto\Lambda(x)$, $x\in L$, extends to a unitary
$U_{l}$ from $\Ltwo M_{\mathrm{left}}$ to $\H$, allowing us to
identify these spaces. We have the useful identity $\left\langle U_{l}^{*}\xi,y\right\rangle _{R^{*},R}=\left\langle \xi,\Lambda(y^{*})\right\rangle _{\H}$
for all $\xi\in\H$ and $y\in R$ \citep[Propositions 2.21, 2.22]{Caspers__Lp_Fourier_LCQG}. 

In the sequel we put $\Linfty M_{\mathrm{left}}:=M$ and $\Lone M_{\mathrm{left}}:=M_{*}$,
and view $M$, $M_{*}$ and $\H$ as linear subspaces of $R^{*}$
by eliminating the usage of $(r^{1})^{*}$, $(r^{\infty})^{*}$ and
$U_{l}^{*}$.

Define $\mathcal{I}:=\left\{ \om\in M_{*}:(\exists\xi(\om)\in\H\;\forall x\in\mathcal{N}_{\varphi})\quad\om(x^{*})=\left\langle \xi(\om),\Lambda(x)\right\rangle \right\} $,
and note that this is precisely $\mathcal{I}$ defined for $\Linfty{\G}$
in the Preliminaries. By \citep[Theorem 3.3]{Caspers__Lp_Fourier_LCQG},
we have $\mathcal{I}=\mathcal{H}\cap M_{*}$ in $R^{*}$, with $\om\in\mathcal{I}$
being equal to $\xi(\om)$. Moreover, the pair $(\H,M_{*})$ is evidently
also compatible. It is proved in \citep[Theorem 3.7]{Caspers__Lp_Fourier_LCQG}
using the reiteration theorem that for $1<p<2$, we have $C_{\frac{2}{p}-1}(\H,M_{*})=L^{p}(M)_{\mathrm{left}}$
in the simplest sense that they are equal as vector subspaces of $R^{*}$
and have the same norm.
\begin{defn}
Let $\G$ be a LCQG. For $1\leq p\leq\infty$, we define $L^{p}(\G)_{\mathrm{left}}$
to be $L^{p}(\Linfty{\G})_{\mathrm{left}}$, calculated with respect
to the left Haar weight $\varphi$. We identify $L^{p}(\G)_{\mathrm{left}}$
with $L^{p}(\G)$ for $p=1,2,\infty$.
\end{defn}
The following generalization of \citep[Theorem 6.4 (i)--(iii)]{Caspers__Lp_Fourier_LCQG}
is proved in the same way, with obvious modifications. For completeness,
we give full details. Handling the last part of the theorem, relating
convolutions and the Fourier transform on non-commutative $L^{p}$-spaces,
requires too much background, and is not needed in this paper. It
is thus left to the reader. A special case of this construction was
developed by Forrest, Lee and Samei \citep[Subsection 6.2]{Forrest_Lee_Samei__projectivity_mod_Fourier}.
\begin{thm}
\label{thm:convolution_Lp}Let $\G$ be a LCQG, $\mu\in\CzU{\G}^{*}$
and $1<p<2$. Consider the maps $\mu*^{1}\in B(\Lone{\G})$, $\Lone{\G}\ni\om\mapsto\mu*\om$,
and $\mu*^{2}:=\l^{\mathrm{u}}(\mu)\in B(\Ltwo{\G})$. Then these
maps are compatible, and the resulting induced operator $\mu*^{p}\in B(L^{p}(\G)_{\mathrm{left}})$
satisfies $\left\Vert \mu*^{p}\right\Vert \leq\left\Vert \mu\right\Vert $.\end{thm}
\begin{proof}
Fix $\om\in\mathcal{I}$. For $\hat{\om}\in\hat{\mathcal{I}}$, write
$y:=\hat{\l}(\hat{\om})\in\Cz{\G}\cap\mathcal{N}_{\varphi}$, and
calculate 
\[
\begin{split}(\mu*\om)(y^{*}) & =(\mu\tensor\om)(\Ww^{*}(\one\tensor y^{*})\Ww)=(\mu\tensor\om)(\Ww^{*}(\one\tensor(\i\tensor\overline{\hat{\om}})(W))\Ww)\\
 & =(\mu\tensor\om\tensor\overline{\hat{\om}})(\Ww_{12}^{*}W_{23}\Ww_{12})=(\mu\tensor\om\tensor\overline{\hat{\om}})(\Ww_{13}W_{23})\\
 & =\overline{\hat{\om}}\left[(\mu\tensor\i)(\Ww)\cdot(\om\tensor\i)(W)\right]=\overline{\hat{\om}\left[(\l^{\mathrm{u}}(\mu)\l(\om))^{*}\right]}\\
 & =\langle\hat{\Lambda}(\l^{\mathrm{u}}(\mu)\l(\om)),\Lambda(y)\rangle=\langle\l^{\mathrm{u}}(\mu)\hat{\Lambda}(\l(\om)),\Lambda(y)\rangle.
\end{split}
\]
As $\hat{\l}(\hat{\mathcal{I}})$ is a core for $\Lambda$, we deduce
that $\mu*\om\in\mathcal{I}$ and $\xi(\mu*\om)=\l^{\mathrm{u}}(\mu)\xi(\om)$
(a slight generalization of \citep[Lemma 4.8]{Van_Daele__LCQGs}).
This means precisely that $\mu*^{1}$ and $\mu*^{2}$ are compatible.

Since $C_{\frac{2}{p}-1}(\Ltwo{\G},\Lone{\G})=L^{p}(\G)_{\mathrm{left}}$
and since $C_{\theta}$ is an exact interpolation functor of exponent
$\theta$, we have the existence of $\mu*^{p}$, and 
\[
\left\Vert \mu*^{p}\right\Vert \leq\left\Vert \mu*^{1}\right\Vert ^{1-((2/p)-1)}\left\Vert \mu*^{2}\right\Vert ^{(2/p)-1}\leq\left\Vert \mu\right\Vert ^{1-((2/p)-1)}\left\Vert \mu\right\Vert ^{(2/p)-1}=\left\Vert \mu\right\Vert .\qedhere
\]
\end{proof}
\begin{rem}
For $p>2$ it may be generally impossible to give a proper meaning
to $\mu*^{p}\om$ when $\mu\in\CzU{\G}^{*}$ and $\om\in L^{p}(\G)_{\mathrm{left}}$.
\end{rem}

\subsection{\label{sub:Lp_duality}Duality}

For $1<p,q<\infty$ with $\frac{1}{p}+\frac{1}{q}=1$, Izumi, generalizing
the classical duality of $L^{p}$-spaces, proved that $L^{p}(\G)_{\mathrm{left}}^{*}\cong L^{q}(\G)_{\mathrm{left}}$
via a natural sesquilinear form $\left(\cdot|\cdot\right)_{p}$ over
$L^{p}(\G)_{\mathrm{left}}\times L^{q}(\G)_{\mathrm{left}}$ (\citep[Theorem 6.1]{Izumi__biln_sesq_forms_duality_noncomm_Lp};
as usual, we are taking $\a=-\frac{1}{2}$ throughout). For $x,y\in L$,
we have $\left(l^{p}(x)|l^{q}(y)\right)_{p}={_{x}}\varphi(y^{*})=\varphi(y^{*}x)$
\citep[Theorem 2.5]{Izumi__biln_sesq_forms_duality_noncomm_Lp}.

If $1<p\leq2$, $\om\in\mathcal{I}=\Lone{\G}\cap\Ltwo{\G}$ and $y\in L$,
then $\left(\om|l^{q}(y)\right)_{p}=\om(y^{*})$. Indeed, endow $\mathcal{I}$
with the natural norm $\left\Vert \om\right\Vert _{\mathcal{I}}:=\max(\left\Vert \om\right\Vert _{\Lone{\G}},\left\Vert \xi(\om)\right\Vert _{\Ltwo{\G}})$,
$\om\in\mathcal{I}$. The embedding $L\hookrightarrow\mathcal{I}$,
$L\ni x\mapsto\presb x{\varphi}$, is contractive with dense range
\citep[Proposition 3.4]{Caspers__Lp_Fourier_LCQG}. If $\left(x_{n}\right)$
is a sequence in $L$ such that $\presb{x_{n}}{\varphi}\to\om$ in
$\mathcal{I}$, then $l^{p}(x_{n})\to\om$ in $L^{p}(\G)_{\mathrm{left}}$,
and so $\left(\om|l^{q}(y)\right)_{p}\leftarrow\left(l^{p}(x_{n})|l^{q}(y)\right)_{p}={_{x_{n}}}\varphi(y^{*})\to\om(y^{*})$.

\section{\label{sec:comparison_of_topologies}Comparison of topologies on
the unit sphere of $\protect\CzU{\protect\G}^{*}$}

In this section we generalize the main results of Granirer and Leinert
\citep{Granirer_Leinert__top_coincide}, and in particular obtain
a result (\prettyref{thm:topologies_in_CPD1}) about positive-definite
functions over LCQGs extending \citep{Raikov__types_of_conv,Yoshizawa__convergence_PD_func}.
\begin{defn}
Let $\G$ be a LCQG. We define several topologies on $\CzU{\G}^{*}$
as follows.
\begin{enumerate}
\item The \emph{strict topology} is the one induced by the semi-norms $\mu\mapsto\left\Vert \om*\mu\right\Vert _{\Lone{\G}}$
and $\mu\mapsto\left\Vert \mu*\om\right\Vert _{\Lone{\G}}$, $\om\in\Lone{\G}$.
\item For $p\in\left[1,2\right]$, the $p$-\emph{strict topology }is the
one induced by the semi-norms $\mu\mapsto\left\Vert \mu*^{p}\om\right\Vert _{p}$,
$\om\in L^{p}(\G)_{\mathrm{left}}$.
\item For $p\in\left[1,2\right]$, a net $\left(\mu_{\be}\right)$ in $\CzU{\G}^{*}$
converges to $\mu\in\CzU{\G}^{*}$ in the \emph{weak} $p$-\emph{strict
topology }if $\mu_{\be}*^{p}\om\to\mu*^{p}\om$ in the $w$-topology
$\sigma(L^{p}(\G)_{\mathrm{left}},L^{p}(\G)_{\mathrm{left}}^{*})$
for every $\om\in L^{p}(\G)_{\mathrm{left}}$.
\item A net $\left(\mu_{\be}\right)$ in $\CzU{\G}^{*}$ converges to $\mu\in\CzU{\G}^{*}$
in $\tau_{nw^{*}}$ if $\mu_{\be}\xrightarrow{w^{*}}\mu$ and $\left\Vert \mu_{\be}\right\Vert \to\left\Vert \mu\right\Vert $.
\item A net $\left(\mu_{\be}\right)$ in $\CzU{\G}^{*}$ converges to $\mu\in\CzU{\G}^{*}$
in $\tau_{bw^{*}}$ if $\mu_{\be}\xrightarrow{w^{*}}\mu$ and $\left(\mu_{\be}\right)$
is bounded.
\end{enumerate}
\end{defn}
We now generalize \citep[Theorem A]{Granirer_Leinert__top_coincide},
answering affirmatively a question raised by Hu, Neufang and Ruan
\citep[p.~140]{Hu_Neufang_Ruan__mod_maps_LCQG}.
\begin{thm}
\label{thm:top_on_C_0_u}Let $\G$ be a LCQG. On $\CzU{\G}^{*}$,
the strict topology is weaker than $\tau_{nw^{*}}$.\end{thm}
\begin{lem}
\label{lem:approx_iden_and_dual}Let $A$ be a $C^{*}$-algebra and
$\left(e_{\a}\right)$ be an approximate identity for $A$. Let $\left(\mu_{\be}\right)$
be a net in $A^{*}$ and $\mu\in A^{*}$ be such that $\mu_{\be}\xrightarrow{w^{*}}\mu$
and $\left\Vert \mu_{\be}\right\Vert \to\left\Vert \mu\right\Vert $.
Then for every $\e>0$ there are $\a_{0},\be_{0}$ such that $\left\Vert e_{\a_{0}}\mu_{\be}-\mu_{\be}\right\Vert <\e$
(resp., $\left\Vert \mu_{\be}e_{\a_{0}}-\mu_{\be}\right\Vert <\e$)
for every $\be\geq\be_{0}$ and $\left\Vert e_{\a_{0}}\mu-\mu\right\Vert <\e$
(resp., $\left\Vert \mu-\mu e_{\a_{0}}\right\Vert <\e$).\end{lem}
\begin{proof}
If $M$ is a von Neumann algebra (e.g., $A^{**}$), recall that the
``absolute value'' of $\nu\in M_{*}$ can be defined in two ways,
as the unique $\left|\nu\right|\in M_{*}^{+}$ with $\left\Vert \left|\nu\right|\right\Vert =\left\Vert \nu\right\Vert $
satisfying either $\left|\nu(x)\right|^{2}\leq\left\Vert \nu\right\Vert \cdot\left|\nu\right|(x^{*}x)$
or $\left|\nu(x)\right|^{2}\leq\left\Vert \nu\right\Vert \cdot\left|\nu\right|(xx^{*})$
for all $x\in M$. We will use the first way to establish half of
the lemma's assertion, the other half being established similarly
using the second way.

For every $\nu\in A^{*}$ and $a\in A$ we have, writing $\one$ for
$\one_{M(A)}$, 
\[
\begin{split}\left|(\nu-e_{\a}\nu)(a)\right|^{2} & =\left|\nu\left(a\left(\one-e_{\a}\right)\right)\right|^{2}\\
 & \leq\left\Vert \nu\right\Vert \left|\nu\right|\left[\left(\one-e_{\a}\right)a^{*}a\left(\one-e_{\a}\right)\right]\\
 & \leq\left\Vert \nu\right\Vert \left\Vert a\right\Vert ^{2}\left|\nu\right|((\one-e_{\a})^{2})\leq\left\Vert \nu\right\Vert \left\Vert a\right\Vert ^{2}\left|\nu\right|(\one-e_{\a}).
\end{split}
\]
Hence $\left\Vert \nu-e_{\a}\nu\right\Vert ^{2}\leq\left\Vert \nu\right\Vert \left|\nu\right|(\one-e_{\a})$.
Since $\left(e_{\a}\right)$ is an approximate identity for $A$,
we have $\left|\nu\right|(\one-e_{\a})\to0$ by strict continuity.
Let $\a_{0}$ be such that $\left\Vert \mu\right\Vert \left|\mu\right|(\one-e_{\a_{0}})<\e^{2}$.
Since $\mu_{\be}\xrightarrow{w^{*}}\mu$ and $\left\Vert \mu_{\be}\right\Vert \to\left\Vert \mu\right\Vert $,
we have $\left|\mu_{\be}\right|\xrightarrow{w^{*}}\left|\mu\right|$
(see Effros \citep[Lemma 3.5]{Effros__order_ideals} or \citep[Proposition III.4.11]{Takesaki__book_vol_1}).
Therefore, 
\[
\left\Vert \mu_{\be}-e_{\a_{0}}\mu_{\be}\right\Vert ^{2}\leq\left\Vert \mu_{\be}\right\Vert \left|\mu_{\be}\right|(\one-e_{\a_{0}})\xrightarrow[\be]{}\left\Vert \mu\right\Vert \left|\mu\right|(\one-e_{\a_{0}})<\e^{2},
\]
so we can choose $\be_{0}$ as asserted.\end{proof}
\begin{lem}
\label{lem:top_on_C_0_u__bw_star__multiplier}Let $a,b\in\CzU{\G}$.
The map $(\CzU{\G}^{*},\tau_{bw^{*}})\to(\CzU{\G}^{*},\text{strict topology})$
given by $\mu\mapsto a\mu b$ is continuous.\end{lem}
\begin{proof}
Let $\left(\mu_{\be}\right)$ be a bounded net in $\CzU{\G}^{*}$
and $\mu\in\CzU{\G}^{*}$ be such that $\mu_{\be}\xrightarrow{w^{*}}\mu$.
Representing $\CzU{\G}$ faithfully on a Hilbert space $\H_{\mathrm{u}}$,
we view the operator $\Ww\in M(\CzU{\G}\tensormin\Cz{\hat{\G}})$
as an element of $B(\H_{\mathrm{u}}\tensor\Ltwo{\G})$. Recall \citep[Proposition 8.3 and its proof]{Kustermans__LCQG_universal}
that for every $\nu\in\CzU{\G}^{*}$ and $\om\in\Cz{\G}^{*}$, the
functional $\nu*\om\in\CzU{\G}^{*}$ corresponds to the element of
$C_{0}(\G)^{*}$ given by 
\[
C_{0}(\G)\ni x\mapsto(\nu\tensor\om)(\Ww^{*}(\one\tensor x)\Ww),
\]
which makes sense because $\Ww^{*}(1\tensor x)\Ww\in M(\CzU{\G}\tensormin C_{0}(\G))$. 

Fix $\om\in\Lone{\G}$, write $\om=\om_{\z,\eta}$ for $\z,\eta\in\Ltwo{\G}$
(this is possible as $\Linfty{\G}$ is in standard form on $\Ltwo{\G}$),
and let $e_{\z},e_{\eta}\in\K(\Ltwo{\G})$ be the projections of $\Ltwo{\G}$
onto $\C\z,\C\eta$, respectively. Then for $\nu\in\CzU{\G}^{*}$,
the functional $(a\nu b)*\om$ corresponds to 
\[
\begin{split}C_{0}(\G)\ni x & \mapsto(\nu\tensor\om_{\z,\eta})((b\tensor\one)\Ww^{*}(\one\tensor x)\Ww(a\tensor\one))\\
 & =(\nu\tensor\om_{\z,\eta})((b\tensor e_{\eta})\Ww^{*}(\one\tensor x)\Ww(a\tensor e_{\z})).
\end{split}
\]
Since $\Ww\in M(\CzU{\G}\tensormin\K(\Ltwo{\G}))$, both $\Ww(a\tensor e_{\z})$
and $(b\tensor e_{\eta})\Ww^{*}$ belong to $\CzU{\G}\tensormin\K(\Ltwo{\G})$.
As a result, approximating them in norm by elements of the corresponding
algebraic tensor product, we see that $\left(\mu_{\be}\right)$ being
bounded and the fact that $\mu_{\be}\xrightarrow{w^{*}}\mu$ imply
that $(a\mu_{\be}b)*\om\xrightarrow{\left\Vert \cdot\right\Vert }(a\mu b)*\om$.
By using the universal version of the unitary antipode $R_{\mathrm{u}}:\CzU{\G}\to\CzU{\G}$
and its properties \citep[Proposition 7.2]{Kustermans__LCQG_universal},
we conclude that also $\om*(a\mu_{\be}b)\xrightarrow{\left\Vert \cdot\right\Vert }\om*(a\mu b)$.
\end{proof}

\begin{proof}[Proof of \prettyref{thm:top_on_C_0_u}]
Let $\left(\mu_{\be}\right)$ be a net in $\CzU{\G}^{*}$ and $\mu\in\CzU{\G}^{*}$
be such that $\mu_{\be}\xrightarrow{nw^{*}}\mu$, and let $\om\in\Lone{\G}$
and $\e>0$. Fix an approximate identity $\left(e_{\a}\right)$ for
$\CzU{\G}$. By invoking \prettyref{lem:approx_iden_and_dual} twice,
we find $\a_{1},\a_{2},\be_{1}$ such that $\left\Vert e_{\a_{1}}\mu_{\be}e_{\a_{2}}-\mu_{\be}\right\Vert <\e$
for every $\be\geq\be_{1}$ and $\left\Vert e_{\a_{1}}\mu e_{\a_{2}}-\mu\right\Vert <\e$.
From \prettyref{lem:top_on_C_0_u__bw_star__multiplier}, there is
$\be_{2}$ such that 
\[
\left\Vert (e_{\a_{1}}\mu_{\be}e_{\a_{2}})*\om-(e_{\a_{1}}\mu e_{\a_{2}})*\om\right\Vert ,\left\Vert \om*(e_{\a_{1}}\mu_{\be}e_{\a_{2}})-\om*(e_{\a_{1}}\mu e_{\a_{2}})\right\Vert <\e
\]
for every $\be\geq\be_{2}$. We conclude that the strict topology
is weaker than $\tau_{nw^{*}}$.
\end{proof}
We now generalize most of \citep[Theorem D]{Granirer_Leinert__top_coincide}
for $1\leq p\leq2$.
\begin{cor}
\label{cor:top_on_C_0_u__Lp}Let $1\leq p\leq2$. On $\CzU{\G}^{*}$,
the $p$-strict topology is weaker than $\tau_{nw^{*}}$, and on bounded
sets, the $w^{*}$-topology is weaker than the weak $p$-strict topology.\end{cor}
\begin{proof}
Let $\left(\mu_{\be}\right)$ be a net in $\CzU{\G}^{*}$ and $\mu\in\CzU{\G}^{*}$.
We use \prettyref{thm:convolution_Lp} and its notation. Suppose that
$\mu_{\be}\xrightarrow{nw^{*}}\mu$. Let $\om\in\mathcal{I}$ and
$\xi:=\xi(\om)$ (so $\om=\xi$ in $R^{*}$). By \prettyref{thm:top_on_C_0_u},
$(\mu_{\be}-\mu)*^{1}\om\to0$ in $\Lone{\G}$. Moreover, $(\mu_{\be}-\mu)*^{2}\xi=\l^{\mathrm{u}}(\mu_{\be}-\mu)\xi\to0$
in $\Ltwo{\G}$ (see \prettyref{thm:top_on_C_0_u__2}, \prettyref{enu:top_on_C_0_u__strict}$\implies$\prettyref{enu:top_on_C_0_u__unif_comp}
below). Since the canonical embedding $(\mathcal{I},\left\Vert \cdot\right\Vert _{\mathcal{I}})\hookrightarrow(L^{p}(\G)_{\mathrm{left}},\left\Vert \cdot\right\Vert _{p})$
is contractive, we infer that $(\mu_{\be}-\mu)*^{p}\om\to0$ in $L^{p}(\G)_{\mathrm{left}}$.
That embedding has dense range and $\left((\mu_{\be}-\mu)*^{p}\right)_{\be}$
is bounded in $B(L^{p}(\G)_{\mathrm{left}})$; hence $(\mu_{\be}-\mu)*^{p}\om\to0$
for all $\om\in L^{p}(\G)_{\mathrm{left}}$.

For the second statement, suppose that $\left(\mu_{\be}\right)$ is
bounded and that $\mu_{\be}\to\mu$ in the weak $p$-strict topology.
We claim that $(\mu_{\be}-\mu)*\om\to0$ in the $w^{*}$-topology
for every $\om\in\Lone{\G}$. Assume for the moment that $p>1$ and
let $q\in[2,\infty)$ be the conjugate of $p$. Let $\om\in\mathcal{I}$
and $y\in L$. If $(\mu_{\be}-\mu)*^{p}\om\to0$ weakly, then by \prettyref{sub:Lp_duality},
we have 
\begin{equation}
((\mu_{\be}-\mu)*\om)(y^{*})=\left((\mu_{\be}-\mu)*^{p}\om|l^{q}(y)\right)_{p}\to0.\label{eq:top_on_C_0_u__Lp__Lp_Lq}
\end{equation}
Denoting by $\mathcal{T}_{\varphi}$ the Tomita algebra of $\varphi$,
the set $\left\{ ab:a,b\in\mathcal{T}_{\varphi}\right\} $ is contained
in $L$ by \citep[Proposition 2.3]{Izumi__noncomm_Lp}. As $\varphi|_{\Cz{\G}_{+}}$
is a $C^{*}$-algebraic KMS weight on $\Cz{\G}$ whose modular automorphism
group is the restriction of that of $\varphi$ to $\Cz{\G}$ \citep[Proposition 1.6 and its proof]{Kustermans_Vaes__LCQG_von_Neumann},
$\mathcal{T}_{\varphi}\cap\Cz{\G}$ is norm dense in $\Cz{\G}$. Hence
$L\cap\Cz{\G}$ is norm dense in $\Cz{\G}$, and $\pi_{\mathrm{u}}^{-1}(L\cap\Cz{\G})$
is norm dense in $\CzU{\G}$. Consequently, \prettyref{eq:top_on_C_0_u__Lp__Lp_Lq}
implies that as elements of $\CzU{\G}^{*}$, $(\mu_{\be}-\mu)*\om\to0$
pointwise on a norm dense subset of $\CzU{\G}$, which, by the boundedness
of $\left(\mu_{\be}\right)$, implies that $(\mu_{\be}-\mu)*\om\to0$
in the $w^{*}$-topology. By density of $\mathcal{I}$ in $\Lone{\G}$
and boundedness again, this holds for every $\om\in\Lone{\G}$, as
claimed. In the case that $p=1$ we have the same result, since the
assumption that $(\mu_{\be}-\mu)*\om\to0$ in the $w$-topology $\sigma(\Lone{\G},\Linfty{\G})$
is formally stronger.

Since $\left\{ (\i\tensor\om)(\Ww^{*}(\one\tensor b)\Ww):\om\in\Lone{\G},b\in C_{0}(\G)\right\} $
is dense in $\CzU{\G}$ and $\left(\mu_{\be}\right)$ is bounded,
we infer from the claim that $\mu_{\be}\to\mu$ in the $w^{*}$-topology.
\end{proof}
Let $G$ be a locally compact group. If $(g_{\be})$ is a bounded
net in $B(G)$ and $g\in B(G)$, then $g_{\be}\to g$ uniformly on
the compact subsets of $G$ if and only if $fg_{\be}\to fg$ in the
$C_{0}(G)$ norm for every $f\in C_{0}(G)$. Indeed, one direction
is trivial, and for the other, notice that $(g_{\be})$ is bounded
in $C_{b}(G)$ since $\left\Vert \cdot\right\Vert _{C_{b}(G)}\leq\left\Vert \cdot\right\Vert _{B(G)}$.
Hence, the following result generalizes \citep[Theorem B$_2$]{Granirer_Leinert__top_coincide}.
\begin{thm}
\label{thm:top_on_C_0_u__2}Let $\G$ be a LCQG and let $S$ denote
the unit sphere of $\CzU{\G}^{*}$. If $\left(\mu_{\be}\right)$ is
a net in $S$ and $\mu\in S$, then the following are equivalent:
\begin{enumerate}
\item \label{enu:top_on_C_0_u__w_star}$\mu_{\be}\to\mu$ in the $w^{*}$-topology;
\item \label{enu:top_on_C_0_u__unif_comp}$\l^{\mathrm{u}}(\mu_{\be})\to\l^{\mathrm{u}}(\mu)$
in the strict topology on $M(\Cz{\hat{\G}})$;
\item \label{enu:top_on_C_0_u__action_on_C_0}$\mu_{\be}\cdot a\to\mu\cdot a$
and $a\cdot\mu_{\be}\to a\cdot\mu$ in $C_{0}(\G)$ for every $a\in C_{0}(\G)$
(see \prettyref{lem:C_0__C_0_u_star});
\item \label{enu:top_on_C_0_u__action_on_L_infty}$\mu_{\be}\cdot a\to\mu\cdot a$
and $a\cdot\mu_{\be}\to a\cdot\mu$ in the $w^{*}$-topology $\sigma(\Linfty{\G},\Lone{\G})$
for every $a\in\Linfty{\G}$, that is: $\mu_{\be}*\om\to\mu*\om$
and $\om*\mu_{\be}\to\om*\mu$ in the $w$-topology $\sigma(\Lone{\G},\Linfty{\G})$
for every $\om\in\Lone{\G}$;
\item \label{enu:top_on_C_0_u__two_products_norm}$(\mu_{\be}*\om)a\to(\mu*\om)a$
and $a(\mu_{\be}*\om)\to a(\mu*\om)$ in $\Lone{\G}$ for every $a\in C_{0}(\G)$,
$\om\in\Lone{\G}$;
\item \label{enu:top_on_C_0_u__two_products_weak}$(\mu_{\be}*\om)a\to(\mu*\om)a$
and $a(\mu_{\be}*\om)\to a(\mu*\om)$ in the $w$-topology $\sigma(\Lone{\G},\Linfty{\G})$
for every $a\in C_{0}(\G)$, $\om\in\Lone{\G}$;
\item \label{enu:top_on_C_0_u__strict}$\mu_{\be}\to\mu$ in the strict
topology;
\item \label{enu:top_on_C_0_u__Lp}for some $1\leq p\leq2$, $\mu_{\be}\to\mu$
in the $p$-strict topology;
\item \label{enu:top_on_C_0_u__Lp_weakly}for some $1\leq p\leq2$, $\mu_{\be}\to\mu$
in the weak $p$-strict topology.
\end{enumerate}
\end{thm}
\begin{proof}
From \prettyref{thm:top_on_C_0_u} and \prettyref{cor:top_on_C_0_u__Lp},
conditions \prettyref{enu:top_on_C_0_u__w_star}, \prettyref{enu:top_on_C_0_u__action_on_L_infty},
\prettyref{enu:top_on_C_0_u__strict}, \prettyref{enu:top_on_C_0_u__Lp}
and \prettyref{enu:top_on_C_0_u__Lp_weakly} are equivalent. It is
clear that \prettyref{enu:top_on_C_0_u__strict}$\implies$ \prettyref{enu:top_on_C_0_u__two_products_norm}$\implies$\prettyref{enu:top_on_C_0_u__two_products_weak}. 

\prettyref{enu:top_on_C_0_u__strict}$\implies$\prettyref{enu:top_on_C_0_u__unif_comp}:
since $\l^{\mathrm{u}}$ is a homomorphism, we have $\l^{\mathrm{u}}(\mu_{\be})\l(\om)\to\l^{\mathrm{u}}(\mu)\l(\om)$
and $\l(\om)\l^{\mathrm{u}}(\mu_{\be})\to\l(\om)\l^{\mathrm{u}}(\mu)$
for every $\om\in\Lone{\G}$. As $\{\l(\om):\om\in\Lone{\G}\}$ is
norm dense in $\Cz{\hat{\G}}$ and $\left(\l^{\mathrm{u}}(\mu_{\be})\right)$
is bounded, we conclude that $\l^{\mathrm{u}}(\mu_{\be})\to\l^{\mathrm{u}}(\mu)$
in the strict topology on $M(\Cz{\hat{\G}})$.

\prettyref{enu:top_on_C_0_u__unif_comp}$\implies$\prettyref{enu:top_on_C_0_u__w_star}:
since $\l^{\mathrm{u}}(\mu_{\be})\to\l^{\mathrm{u}}(\mu)$ in the
strict topology on $M(\Cz{\hat{\G}})$ and $\left(\l^{\mathrm{u}}(\mu_{\be})\right)$
is bounded, this convergence holds in the ultraweak topology as well.
So for all $\hat{\om}\in\Lone{\hat{\G}}$, we have 
\[
(\mu_{\be}-\mu)((\i\tensor\hat{\om})(\Ww))=\hat{\om}(\l^{\mathrm{u}}(\mu_{\be}-\mu))\to0.
\]
As $\{(\i\tensor\hat{\om})(\Ww):\hat{\om}\in\Lone{\hat{\G}}\}$ is
dense in $\CzU{\G}$ and $\left(\mu_{\be}\right)$ is bounded, we
infer that $\mu_{\be}\to\mu$ in the $w^{*}$-topology.

\prettyref{enu:top_on_C_0_u__strict}$\implies$\prettyref{enu:top_on_C_0_u__action_on_C_0}:
we may assume that $a=\om\cdot b$ for some $\om\in\Lone{\G}$ and
$b\in C_{0}(\G)$, because the set of these elements spans a dense
subset of $C_{0}(\G)$. Hence 
\[
(\mu_{\be}-\mu)\cdot a=(\mu_{\be}-\mu)\cdot(\om\cdot b)=((\mu_{\be}-\mu)*\om)\cdot b\to0,
\]
and similarly $a\cdot(\mu_{\be}-\mu)\to0$.

The proofs of \prettyref{enu:top_on_C_0_u__action_on_C_0}$\implies$\prettyref{enu:top_on_C_0_u__w_star}
and \prettyref{enu:top_on_C_0_u__two_products_weak}$\implies$\prettyref{enu:top_on_C_0_u__w_star}
are left to the reader (see the proof of \prettyref{cor:top_on_C_0_u__Lp},
and use that $\Cz{\G}^{2}=\Cz{\G}$).\end{proof}
\begin{rem}
In view of \prettyref{thm:top_on_C_0_u__2}, the following is noteworthy.
Let $\G$ be a compact quantum group. Generalizing a classical result
about discrete groups, Kyed \citep[Theorem 3.1]{Kyed__cohom_prop_T_QG}
proved that the discrete dual $\hat{\G}$ has property (T) if and
only if every net of states of $C^{\mathrm{u}}(\G)$, converging in
the $w^{*}$-topology to the co-unit, converges in norm.
\end{rem}
A classical result \citep{Raikov__types_of_conv,Yoshizawa__convergence_PD_func}
says that if $G$ is a locally compact group, then on the set of positive-definite
functions of $\Linfty G$-norm $1$, the $w^{*}$-topology $\sigma(\Linfty G,\Lone G)$
and the topology of uniform convergence on compact subsets coincide.
The following generalizes this to LCQGs.
\begin{thm}
\label{thm:topologies_in_CPD1}Assume that $\G$ is a co-amenable
LCQG. On the subset $S$ of $M(C_{0}(\G))$ consisting of all positive-definite
elements of norm $1$, the strict topology induced by $C_{0}(\G)$
coincides with the weak and the strong operator topologies on $\Ltwo{\G}$.\end{thm}
\begin{proof}
By co-amenability, the map $\hat{\mu}\mapsto\hat{\l}^{\mathrm{u}}(\hat{\mu})$
is an isometric isomorphism between the unit sphere of $\CzU{\hat{\G}}_{+}^{*}$
and $S$ (\prettyref{thm:pos_def} and \prettyref{rem:CPD_norm}).
Now apply \prettyref{thm:top_on_C_0_u__2}, \prettyref{enu:top_on_C_0_u__w_star}$\iff$\prettyref{enu:top_on_C_0_u__unif_comp},
to $\hat{\G}$ in place of $\G$, and notice that for a bounded net
in $\CzU{\hat{\G}}^{*}$, $w^{*}$-convergence is equivalent to convergence
in the weak operator topology of its image under $\hat{\l}^{\mathrm{u}}$.
Moreover, on bounded sets, the strict topology on $M(C_{0}(\G))$
is finer than the strong operator topology.\end{proof}
\begin{rem}
Attempting to prove \prettyref{thm:topologies_in_CPD1} by generalizing
the proof of \citep[Theorem 13.5.2]{Dixmier__C_star_English} yielded
only partially successful: we were able to establish that on $S$,
the weak operator topology coincides with the topology on $M(C_{0}(\G))$
in which a net $\left(x_{\be}\right)$ converges to $x$ if and only
if $yx_{\be}z\to yxz$ for every $y,z\in C_{0}(\G)$. This topology
evidently coincides with the strict one when $\G$ is commutative,
but not generally.

However, it is worth mentioning that taking this approach, one encounters
a straightforward generalization of a very useful inequality, namely
that if $\varphi$ is a (continuous) positive definite function on
a locally compact group $G$, then $\left|\varphi(s)-\varphi(t)\right|^{2}\leq2\varphi(e)(\varphi(e)-\mathrm{Re}\varphi(s^{-1}t))$
for every $s,t\in G$ \citep[Proposition 13.4.7]{Dixmier__C_star_English}.
As $\varphi(s^{-1})=\overline{\varphi(s)}$, that is equivalent to
$\left|\varphi(st)-\varphi(t)\right|^{2}\leq2\varphi(e)(\varphi(e)-\mathrm{Re}\varphi(s))$
for every $s,t\in G$. If $\G$ is a co-amenable LCQG and $y$ is
positive definite over $\G$, write $y=(\i\tensor\hat{\mu})(\wW^{*})$
for a suitable $\hat{\mu}\in\CzU{\hat{\G}}_{+}^{*}$. Now $\Delta(y)-\one\tensor y=(\i\tensor\i\tensor\hat{\mu})(\wW_{23}^{*}(\wW_{13}^{*}-\one))$,
and as $\i\tensor\i\tensor\hat{\mu}$ is a completely positive map
of $cb$-norm $\left\Vert \hat{\mu}\right\Vert =\left\Vert y\right\Vert $,
the Kadison--Schwarz inequality implies that
\begin{equation}
\begin{split}\left[\Delta(y)-\one\tensor y\right]^{*}\left[\Delta(y)-\one\tensor y\right] & \leq\left\Vert y\right\Vert (\i\tensor\i\tensor\hat{\mu})((\wW_{13}-\one)\wW_{23}\wW_{23}^{*}(\wW_{13}^{*}-\one))\\
 & =\left\Vert y\right\Vert (\i\tensor\i\tensor\hat{\mu})(2\one-\wW_{13}-\wW_{13}^{*})\\
 & =\left\Vert y\right\Vert \left[2\left\Vert y\right\Vert \one-(y^{*}+y)\right]\tensor\one.
\end{split}
\label{eq:pos_def_ineq}
\end{equation}

\end{rem}

\section{\label{sec:char_co_amen_of_dual}A characterization of co-amenability
of the dual}

Related to the notion of a positive-definite function is the notion
of a (generally unbounded) positive-definite measure (\citep{Godement__positive_def_func},
\citep[Section 13.7]{Dixmier__C_star_English}). The purpose of this
section is to generalize a classical result of Godement connecting
amenability to positive definiteness (\citep[Proposition 18.3.6]{Dixmier__C_star_English},
originally \citep[pp.~76--77]{Godement__positive_def_func}, see also
Valette \citep{Valette__Godement_amen}).
\begin{defn}
An element $\mu\in\Cz{\G}^{*}$ is called a \emph{bounded positive-definite
measure on $\G$} if $\l(\mu)$ is positive in $M(\Cz{\hat{\G}})$.\end{defn}
\begin{thm}
\label{thm:co_amen__char}Let $\G$ be a co-amenable LCQG. The following
conditions are equivalent:
\begin{enumerate}
\item \label{enu:co_amen__char__1}$\hat{\G}$ is co-amenable;
\item \label{enu:co_amen__char__2}every positive-definite function on $\G$
is the strict limit in $M(\Cz{\G})$ of a bounded net of positive-definite
functions in $\hat{\l}(\Lone{\hat{\G}}_{+})\cap\mathcal{N}_{\varphi}$;
\item \label{enu:co_amen__char__3}every positive-definite function on $\G$
is the strict limit in $M(\Cz{\G})$ of a bounded net of positive-definite
functions in $\hat{\l}(\Lone{\hat{\G}}_{+})$;
\item \label{enu:co_amen__char__4}$\mu(x^{*})\geq0$ for every bounded
positive-definite measure $\mu$ on $\G$ and every positive-definite
function $x$ on $\G$;
\item \label{enu:co_amen__char__5}$\mu(\one_{M(\Cz{\G})})\geq0$ for every
bounded positive-definite measure $\mu$ on $\G$.
\end{enumerate}
\end{thm}
\begin{lem}
\label{lem:cone_Q_is_closed}Let $\G$ be a co-amenable LCQG. Then
the cone $Q:=\hat{\l}(\Cz{\hat{\G}}_{+}^{*})$ is ultraweakly closed
in $\Linfty{\G}$.\end{lem}
\begin{proof}
By the Krein--\v{S}mulian theorem, it suffices to prove that $Q_{1}$,
the intersection of $Q$ with the closed unit ball of $\Linfty{\G}$,
is ultraweakly closed. Let $\left(x_{\a}\right)$ be a net in $Q_{1}$
converging ultraweakly to some $x\in\Linfty{\G}$. Write $x_{\a}=\hat{\l}(\hat{\mu}_{\a})$,
$\hat{\mu}_{\a}\in\Cz{\hat{\G}}_{+}^{*}$, for every $\a$. By \prettyref{rem:CPD_norm},
$\left(\hat{\mu}_{\a}\right)$ is bounded by one, and so it has a
subnet converging in the $w^{*}$-topology to some $\hat{\mu}\in\Cz{\hat{\G}}_{+}^{*}$.
Hence $x=\hat{\l}^{\mathrm{u}}(\hat{\mu})\in Q_{1}$.
\end{proof}

\begin{proof}[Proof of \prettyref{thm:co_amen__char}]
\prettyref{enu:co_amen__char__1}$\implies$\prettyref{enu:co_amen__char__2}:
every positive-definite function has the form $\hat{\l}(\hat{\nu})$
for some $\hat{\nu}\in\CzU{\hat{\G}}_{+}^{*}=\Cz{\hat{\G}}_{+}^{*}$
by co-amenability of $\hat{\G}$ (\prettyref{thm:pos_def}). Now $\hat{\nu}$
is the $w^{*}$-limit of a bounded net $\left(\hat{\om}_{\be}\right)$
in $\Lone{\hat{\G}}_{+}$. Since each element of $\Lone{\hat{\G}}_{+}$
can be approximated in norm by elements of $\hat{\mathcal{I}}_{+}$
of the same norm \citep[Lemma 4.7]{Van_Daele__LCQGs}, we may assume
that $\hat{\om}_{\be}\in\hat{\mathcal{I}}$, and hence $\hat{\l}(\hat{\om}_{\be})\in\mathcal{N}_{\varphi}$,
for every $\be$. From \prettyref{thm:top_on_C_0_u__2} applied to
$\hat{\G}$, we infer that $\hat{\l}(\hat{\om}_{\be})\to\hat{\l}(\hat{\nu})$
strictly in $M(\Cz{\G})$. 

\prettyref{enu:co_amen__char__2}$\implies$\prettyref{enu:co_amen__char__3}:
clear.

\prettyref{enu:co_amen__char__3}$\implies$\prettyref{enu:co_amen__char__4}:
let $\mu$ be a bounded positive-definite measure on $\G$. For every
$\hat{\om}\in\Lone{\hat{\G}}_{+}$, 
\[
\overline{\mu}(\hat{\lambda}(\hat{\om}))=(\overline{\mu}\tensor\hat{\om})(W^{*})=\hat{\om}(\l(\mu)^{*})\geq0.
\]
If $x$ is a positive-definite function on $\G$ and $\left(\hat{\om}_{\be}\right)$
is a net in $\Lone{\hat{\G}}_{+}$ such that $\hat{\l}(\hat{\om}_{\be})\to x$
strictly in $M(\Cz{\G})$, then $\overline{\mu}(\hat{\l}(\hat{\om}_{\be}))\to\overline{\mu}(x)=\overline{\mu(x^{*})}$.
Hence $\overline{\mu(x^{*})}$, or equivalently $\mu(x^{*})$, is
non-negative.

\prettyref{enu:co_amen__char__4}$\implies$\prettyref{enu:co_amen__char__5}:
trivial, as $\one:=\one_{M(\Cz{\G})}=\hat{\l}^{\mathrm{u}}(\hat{\epsilon})$
is positive definite.

\prettyref{enu:co_amen__char__5}$\implies$\prettyref{enu:co_amen__char__1}:
as $\hat{\l}^{\mathrm{u}}$ is injective, we should establish that
$\one$ belongs to $Q$. By \prettyref{lem:cone_Q_is_closed}, $Q$
is an ultraweakly closed cone, so it is enough to show that $\one$
belongs to the bipolar of $Q$. Here we are using the version of the
bipolar theorem in which the pre-polar of $Q$ is given by $Q_{\circ}:=\left\{ \om\in\Lone{\G}:(\forall x\in Q)\quad0\leq\Ree\om(x)\right\} $,
and its polar is defined similarly. Note that $Q$ is invariant under
the scaling group, as $\tau_{t}(\hat{\l}(\hat{\mu}))=\hat{\l}(\hat{\mu}\circ\hat{\tau}_{-t})$
for every $\hat{\mu}\in\Cz{\hat{\G}}^{*}$, $t\in\R$ \citep[Propositions 8.23 and 8.25]{Kustermans_Vaes__LCQG_C_star}.
Consequently, 
\[
V:=Q_{\circ}\cap D((\tau_{*})_{-i/2})
\]
is norm dense in $Q_{\circ}$ by a standard smearing argument (e.g.,
see \citep[proof of Proposition 5.26]{Kustermans_Vaes__LCQG_C_star}).
So picking $\om_{0}\in V$, we should show that $0\leq\Ree\om_{0}(\one)$.
For every $\hat{\nu}\in\Cz{\hat{\G}}_{+}^{*}$ we have 
\[
0\leq\Ree\om_{0}(\hat{\l}(\hat{\nu}))=\Ree\overline{\om_{0}(\hat{\l}(\hat{\nu}))}=\Ree(\overline{\om_{0}}\tensor\hat{\nu})(W)=\Ree\hat{\nu}(\l(\overline{\om_{0}})).
\]
Thus $0\leq\l(\overline{\om_{0}})+\l(\overline{\om_{0}})^{*}=\l(\overline{\om_{0}}+\overline{\om_{0}}^{*})$,
that is: $\overline{\om_{0}}+\overline{\om_{0}}^{*}$, as an element
of $\Lone{\G}\hookrightarrow\Cz{\G}^{*}$, is a bounded positive-definite
measure. By assumption, $0\leq(\overline{\om_{0}}+\overline{\om_{0}}^{*})(\one)=(\overline{\om_{0}}+\om_{0})(\one)=2\Ree\om_{0}(\one)$
as $\one\in D(S)$ and $S(\one)=\one$. In conclusion, $\one$ belongs
to the bipolar of $Q$.
\end{proof}

\section{\label{sec:separation_property}The separation property}

\subsection{Preliminaries}
\begin{defn}[Lau and Losert \citep{Lau_Losert__w_star_closed_comp_inv_subsp},
Kaniuth and Lau \citep{Kaniuth_Lau__separation_prop}]
Let $G$ be a locally compact group and $H$ be a closed subgroup
of $G$. We say that $G$ has the \emph{$H$-separation property}
if for every $g\in G\backslash H$ there exists a positive-definite
function $\varphi$ on $G$ with $\varphi|_{H}\equiv1$ but $\varphi(g)\neq1$.
\end{defn}
It was first observed in \citep{Lau_Losert__w_star_closed_comp_inv_subsp}
that $G$ has the $H$-separation property if $H$ is either normal,
compact or open. Generalizing a result of Forrest \citep{Forrest__amen_ideals_A_G},
it was proved that $G$ has the $H$-separation property provided
that $G$ has small $H$-invariant neighborhoods \citep[Proposition 2.2]{Kaniuth_Lau__separation_prop}.
The property was subsequently explored further in several papers,
including \citep{Kaniuth_Lau__separation_prop2,Kaniuth_Lau__ext_sep_prop}.
It is somewhat related to another property connecting positive-definite
functions and closed subgroups, namely the extension property.

In this section we introduce the separation property for LCQGs and
obtain a first result about it. To this end, we continue with some
background on closed quantum subgroups of LCQGs. To simplify the notation
a little, throughout this section we will use $\pi$ for the surjection
$\pi_{\mathrm{u}}:\CzU{\G}\to\Cz{\G}$, $\G$ being a LCQG.
\begin{defn}[Meyer, Roy and Woronowicz \citep{Meyer_Roy_Woronowicz__hom_quant_grps}]
Let $\G,\HH$ be LCQGs. A \emph{strong quantum homomorphism from
$\HH$ to $\G$} is a nondegenerate $*$-homomorphism $\Phi:\CzU{\G}\to M(C_{0}^{\mathrm{u}}(\HH))$
such that $(\Phi\tensor\Phi)\circ\Delta_{\G}^{\mathrm{u}}=\Delta_{\HH}^{\mathrm{u}}\circ\Phi$.
\end{defn}
Every such $\Phi$ has a dual object \citep[Proposition 3.9 and Theorem 4.8]{Meyer_Roy_Woronowicz__hom_quant_grps},
which is the (unique) strong quantum homomorphism $\hat{\Phi}$ from
$\hat{\G}$ to $\hat{\HH}$ that satisfies 
\begin{equation}
(\Phi\tensor\i)(\WW_{\G})=(\i\tensor\hat{\Phi})(\WW_{\HH})\label{eq:quantum_hom_and_dual}
\end{equation}
(here and in the sequel we use the left version of this theory, in
contrast to \citep{Daws_Kasprzak_Skalski_Soltan__closed_q_subgroups_LCQGs,Meyer_Roy_Woronowicz__hom_quant_grps},
which use the right one). As customary, we will write $\Phi$ also
for its unique extension to a $*$-homomorphism $M(\CzU{\G})\to M(C_{0}^{\mathrm{u}}(\HH))$.
\begin{defn}[{Daws, Kasprzak, Skalski and So{\ldash}tan \citep[Definitions 3.1, 3.2 and Theorems 3.3, 3.6]{Daws_Kasprzak_Skalski_Soltan__closed_q_subgroups_LCQGs}}]
Let $\G,\HH$ be LCQGs. 
\begin{enumerate}
\item We say that $\HH$ is a \emph{closed quantum subgroup} \emph{of $\G$
in the sense of Vaes} if there exists a faithful normal $*$-homomorphism
$\gamma:\Linfty{\hat{\HH}}\to\Linfty{\hat{\G}}$ such that $(\gamma\tensor\gamma)\circ\Delta_{\hat{\HH}}=\Delta_{\hat{\G}}\circ\gamma$.
\item We say that $\HH$ is a \emph{closed quantum subgroup} \emph{of $\G$
in the sense of Woronowicz} if there exists a strong quantum homomorphism
$\Phi$ from $\HH$ to $\G$ such that $\Phi(\CzU{\G})=\CzU{\HH}$.
\end{enumerate}
\end{defn}
A fundamental result \citep[Theorem 3.5]{Daws_Kasprzak_Skalski_Soltan__closed_q_subgroups_LCQGs}
is that if $\HH$ is a closed quantum subgroup of $\G$ in the sense
of Vaes, then it is also a closed quantum subgroup of $\G$ in the
sense of Woronowicz. In this case, the maps $\gamma$ and $\Phi$
are related by the identity $\gamma|_{\Cz{\hat{\HH}}}\circ\pi_{\hat{\HH}}=\pi_{\hat{\G}}\circ\hat{\Phi}$.
The converse is true if $\G$ is either commutative, co-commutative
or discrete, or if $\HH$ is compact \citep[Sections 4--6]{Daws_Kasprzak_Skalski_Soltan__closed_q_subgroups_LCQGs}.

\subsection{The separation property for LCQGs}
\begin{defn}
\label{def:sep_prop}Let $\G$ be a LCQG and $\HH$ be a closed quantum
subgroup of $\G$ in the sense of Woronowicz via a strong quantum
homomorphism $\Phi:\CzU{\G}\to\CzU{\HH}$. We say that $\G$ has the
\emph{$\HH$-separation property} if whenever $\mu\in\CzU{\G}_{+}^{*}$
is a state such that $(\mu\tensor\i)(\WW_{\G})\notin\hat{\Phi}(M(\CzU{\hat{\HH}}))$,
there is $\hat{\om}\in\CzU{\hat{\G}}_{+}^{*}$ so that $\Phi((\i\tensor\hat{\om})(\WW_{\G}))=\one_{M(\CzU{\HH})}$
but $\mu((\i\tensor\hat{\om})(\WW_{\G}))\neq1$.
\end{defn}
If $\G$ (thus $\HH$) is commutative, this definition reduces to
the classical one. Generally, for $\hat{\om}\in\CzU{\hat{\G}}_{+}^{*}$,
note that $\Phi((\i\tensor\hat{\om})(\WW_{\G}))=(\i\tensor(\hat{\om}\circ\hat{\Phi}))(\WW_{\HH})$
by \prettyref{eq:quantum_hom_and_dual} and $(\i\tensor\hat{\epsilon}_{\HH})(\WW_{\HH})=\one_{M(\CzU{\HH})}$,
hence the equality $\Phi((\i\tensor\hat{\om})(\WW_{\G}))=\one_{M(\CzU{\HH})}$
is equivalent to $\hat{\om}\circ\hat{\Phi}=\hat{\epsilon}_{\HH}$. 
\begin{thm}
\label{thm:compact_sep_prop}Let $\G$ be a LCQG and $\HH$ a compact
quantum subgroup of $\G$. Let $\hat{p}$ be the central minimal projection
in $\ell^{\infty}(\hat{\HH})$ with $\hat{a}\hat{p}=\hat{\epsilon}_{\HH}(\hat{a})\hat{p}=\hat{p}\hat{a}$
for every $\hat{a}\in\ell^{\infty}(\hat{\HH})$, and assume that the
following condition holds:
\begin{equation}
\text{for every }\hat{z}\in M(\Cz{\hat{\G}})\text{, if }\hat{\Delta}_{\G}(\hat{z})(\gamma(\hat{p})\tensor\one)=\gamma(\hat{p})\tensor\hat{z}\text{ then }\hat{z}\in\Img\gamma.\label{eq:compact_sep_prop_condition}
\end{equation}
Then $\G$ has the $\HH$-separation property.
\end{thm}
It will be clear from the proof of \prettyref{thm:compact_sep_prop}
that a condition weaker than \prettyref{eq:compact_sep_prop_condition}
is enough. However, \prettyref{eq:compact_sep_prop_condition} is
often easier to check.

Before proving the theorem, observe that each $\hat{z}\in\Img\gamma$
indeed satisfies $\hat{\Delta}_{\G}(\hat{z})(\gamma(\hat{p})\tensor\one)=\gamma(\hat{p})\tensor\hat{z}$
(see Van Daele \citep[Proposition 3.1]{Van_Daele__discrete_CQs}).
Also, if $\hat{z}\in\Linfty{\hat{\G}}$ satisfies this identity, then
taking $\hat{\om}\in\Lone{\hat{\G}}$ with $\hat{\om}(\gamma(\hat{p}))=1$,
we get $(\gamma(\hat{p})\hat{\om}\tensor\i)\hat{\Delta}_{\G}(\hat{z})=\hat{z}$,
so in the terminology of \citep{Runde__unif_cont_LCQG}, we have $\hat{z}\in\mathrm{LUC}(\hat{\G})$,
thus $\hat{z}\in M(\Cz{\hat{\G}})$ \citep[Theorem 2.4]{Runde__unif_cont_LCQG}.
Furthermore, if $\G$ is commutative or co-commutative, then \prettyref{eq:compact_sep_prop_condition}
holds automatically by \citep[Sections 4, 5]{Daws_Kasprzak_Skalski_Soltan__closed_q_subgroups_LCQGs};
we prove the former case below, and the second one, in which $\G=\hat{G}$
for some locally compact group $G$ and $\HH=\widehat{G/A}$ for an
open normal subgroup $A$ of $G$, is a simple observation. At the
moment it is unclear whether \prettyref{eq:compact_sep_prop_condition}
always holds, but we will show in \prettyref{sub:bicrossed_products}
that it holds in an abundance of examples in which closed quantum
subgroups appear naturally, namely via the bicrossed product construction,
and in \prettyref{sub:E_2} that it holds for $\mathbb{T}$ as a closed
quantum subgroup of quantum $E(2)$.
\begin{prop}
Condition \prettyref{eq:compact_sep_prop_condition} holds when $\G$
is commutative.\end{prop}
\begin{proof}
Let $G$ be a locally compact group and $H$ a compact subgroup of
$G$. The embedding $\gamma:\VN H\to M(C_{r}^{*}(G))\subseteq\VN G$
is the natural one, mapping $\l_{h}\in\VN H$, $h\in H$, to $\l_{h}$
in $\VN G$. Also $\gamma(\hat{p})=\int_{H}\l_{h}\d h$. Replacing
$\hat{z}$ by its adjoint in \prettyref{eq:compact_sep_prop_condition},
suppose that $\hat{z}\in\VN G$ and $(\gamma(\hat{p})\tensor\one)\hat{\Delta}(\hat{z})=\gamma(\hat{p})\tensor\hat{z}$.
Denote by $\ell_{t}$, $t\in G$, the left shift operators over $A(G)$.
For all $\om_{1},\om_{2}\in A(G)$ and $t\in G$, one calculates that
\[
(\om_{1}\tensor\om_{2})[(\l_{t}\tensor\one)\hat{\Delta}(\hat{z})]=(\ell_{t^{-1}}(\om_{1})\cdot\om_{2})(\hat{z}),
\]
and thus 
\[
(\om_{1}\tensor\om_{2})[(\gamma(\hat{p})\tensor\one)\hat{\Delta}(\hat{z})]=\int_{H}(\ell_{h^{-1}}(\om_{1})\cdot\om_{2})(\hat{z})\d h=\left((\int_{H}\ell_{h^{-1}}(\om_{1})\d h)\cdot\om_{2}\right)(\hat{z})
\]
(the second integral is in the norm of $A(G)$), and by assumption
it is equal to
\[
(\om_{1}\tensor\om_{2})(\gamma(\hat{p})\tensor\hat{z})=\int_{H}\om_{1}(h)\d h\cdot\om_{2}(\hat{z}).
\]
Fix a closed set $C$ with $C\cap H=\emptyset$ and $\om_{2}\in A(G)$
that is supported by $C$. Noticing that $HC\cap H=\emptyset$, let
$\om_{1}\in A(G)$ be such that $\om_{1}|_{H}\equiv1$ and $\om_{1}|_{HC}\equiv0$
\citep[Lemme 3.2]{Eymard__Fourier_alg}. We have
\[
0=\left((\int_{H}\ell_{h^{-1}}(\om_{1})\d h)\cdot\om_{2}\right)(\hat{z})=\int_{H}\om_{1}(h)\d h\cdot\om_{2}(\hat{z})=\om_{2}(\hat{z}).
\]
Consequently, the support of $\hat{z}$ (see Eymard \citep[D\'{e}finition 4.5 and Proposition 4.8]{Eymard__Fourier_alg})
is contained in $H$. Consequently, by Takesaki and Tatsuuma \citep{Takesaki_Tatsuuma__duality_subgroups_2},
$\hat{z}$ belongs to $\gamma(\VN H)$, as desired.
\end{proof}

\begin{proof}[Proof of \prettyref{thm:compact_sep_prop}]
Let $\mu\in\CzU{\G}_{+}^{*}$ be a state such that $(\mu\tensor\i)(\WW_{\G})\notin\hat{\Phi}(M(c_{0}(\hat{\HH})))$.
We should prove that there exists $\hat{\om}\in\CzU{\hat{\G}}_{+}^{*}$
so that $\hat{\om}\circ\hat{\Phi}=\hat{\epsilon}_{\HH}$ but $\mu((\i\tensor\hat{\om})(\WW_{\G}))\neq1$.
Assume by contradiction that $\mu((\i\tensor\hat{\om})(\WW_{\G}))=1$
for every $\hat{\om}\in\CzU{\hat{\G}}_{+}^{*}$ such that $\hat{\om}\circ\hat{\Phi}=\hat{\epsilon}_{\HH}$.
Representing $M(\CzU{\hat{\G}})$ faithfully on some Hilbert space,
every unit vector $\z\in\Img\hat{\Phi}(\hat{p})$ satisfies $\hat{\om}_{\z}\circ\hat{\Phi}=\hat{\epsilon}_{\HH}$.
Hence $\hat{\om}_{\z}\left[(\mu\tensor\i)(\WW_{\G})\right]=1$ for
every such vector, and as $\left\Vert (\mu\tensor\i)(\WW_{\G})\right\Vert =1$,
we obtain 
\[
(\mu\tensor\i)(\WW_{\G})\hat{\Phi}(\hat{p})=\hat{\Phi}(\hat{p})=\hat{\Phi}(\hat{p})(\mu\tensor\i)(\WW_{\G}).
\]
Denote $\hat{y}:=(\mu\tensor\i)(\WW_{\G})\in M(\CzU{\hat{\G}})$.
Since $\mu$ is a state, a variant of \prettyref{eq:pos_def_ineq}
implies that 
\[
\bigl[\hat{\Delta}_{\G}^{\mathrm{u}}(\hat{y})-\one\tensor\hat{y}\bigr]^{*}\bigl[\hat{\Delta}_{\G}^{\mathrm{u}}(\hat{y})-\one\tensor\hat{y}\bigr]\leq\left[2\one-(\hat{y}^{*}+\hat{y})\right]\tensor\one.
\]
Multiplying by $\hat{\Phi}(\hat{p})\tensor\one$ on both sides we
get $\bigl[\hat{\Delta}_{\G}^{\mathrm{u}}(\hat{y})-\one\tensor\hat{y}\bigr](\hat{\Phi}(\hat{p})\tensor\one)=0$,
that is, $\hat{\Delta}_{\G}^{\mathrm{u}}(\hat{y})(\hat{\Phi}(\hat{p})\tensor\one)=\hat{\Phi}(\hat{p})\tensor\hat{y}$.
Applying $\pi_{\hat{\G}}\tensor\pi_{\hat{\G}}$ to both sides and
using that $\pi_{\hat{\G}}\circ\hat{\Phi}=\gamma$, we get $\hat{\Delta}_{\G}(\pi_{\hat{\G}}(\hat{y}))(\gamma(\hat{p})\tensor\one)=\gamma(\hat{p})\tensor\pi_{\hat{\G}}(\hat{y})$.
By \prettyref{eq:compact_sep_prop_condition}, 
\[
(\mu\tensor\i)(\Ww_{\G})=\pi_{\hat{\G}}(\hat{y})\in(\pi_{\hat{\G}}\circ\hat{\Phi})(M(c_{0}(\hat{\HH}))).
\]
From \prettyref{lem:compact_quant_subgrp_pi_G_hat} below we obtain
$(\mu\tensor\i)(\WW_{\G})\in\hat{\Phi}(M(c_{0}(\hat{\HH})))$, a contradiction.\end{proof}
\begin{lem}
\label{lem:compact_quant_subgrp_pi_G_hat}Let $\G$ be a LCQG and
$\HH$ be a compact quantum subgroup of $\G$. If $\mu\in\CzU{\G}^{*}$
is such that $\hat{x}:=(\mu\tensor\i)(\WW_{\G})$ satisfies $\pi_{\hat{\G}}(\hat{x})\in(\pi_{\hat{\G}}\circ\hat{\Phi})(M(c_{0}(\hat{\HH})))$,
then $\hat{x}\in\hat{\Phi}(M(c_{0}(\hat{\HH})))$.\end{lem}
\begin{proof}
Recall that up to isomorphism, $c_{0}(\hat{\HH})$ decomposes as $c_{0}-\bigoplus_{\a\in\mathrm{Irred}(\HH)}M_{n(\a)}$.
For each $\a\in\mathrm{Irred}(\HH)$, write $\hat{p}_{\a}\in c_{0}(\hat{\HH})$
for the identity of $M_{n(\a)}$, and let $\om_{\a}\in C^{\mathrm{u}}(\HH)^{*}$
be such that $\hat{p}_{\a}=(\om_{\a}\tensor\i)(\Ww_{\HH})$ (which
exists by the Peter--Weyl theory for compact quantum groups \citep{Woronowicz__symetries_quantiques}).
Then $(\pi_{\hat{\G}}\circ\hat{\Phi})(\hat{p}_{\a})=((\om_{\a}\circ\Phi)\tensor\i)(\Ww_{\G})$
by \prettyref{eq:quantum_hom_and_dual}, and 
\[
(\mu\tensor\i)(\Ww_{\G})\cdot(\pi_{\hat{\G}}\circ\hat{\Phi})(\hat{p}_{\a})\in(\pi_{\hat{\G}}\circ\hat{\Phi})(M(c_{0}(\hat{\HH}))).
\]
If $\hat{y}_{\a}\in M_{n(\a)}$ is such that $(\mu\tensor\i)(\Ww_{\G})\cdot(\pi_{\hat{\G}}\circ\hat{\Phi})(\hat{p}_{\a})=(\pi_{\hat{\G}}\circ\hat{\Phi})(\hat{y}_{\a})$,
there exists $\rho_{\a}\in C^{\mathrm{u}}(\HH)^{*}$ with $\hat{y}_{\a}=(\rho_{\a}\tensor\i)(\Ww_{\HH})$.
Thus 
\[
((\mu*(\om_{\a}\circ\Phi))\tensor\i)(\Ww_{\G})=((\rho_{\a}\circ\Phi)\tensor\i)(\Ww_{\G}).
\]
Hence $\mu*(\om_{\a}\circ\Phi)=\rho_{\a}\circ\Phi$ as $\l_{\G}^{\mathrm{u}}$
is injective, and we can replace $\Ww_{\G}$ by $\WW_{\G}$ to obtain
\[
(\mu\tensor\i)(\WW_{\G})\cdot\hat{\Phi}(\hat{p}_{\a})\in\hat{\Phi}(M(c_{0}(\hat{\HH}))).
\]
But $\sum_{\a\in\mathrm{Irred}(\HH)}\hat{\Phi}(\hat{p}_{\a})=\one$
strictly in $M(\CzU{\hat{\G}})$ since $\sum_{\a\in\mathrm{Irred}(\HH)}\hat{p}_{\a}=\one$
strictly in $M(c_{0}(\hat{\HH}))$ and $\hat{\Phi}$ is nondegenerate,
so we conclude that $(\mu\tensor\i)(\WW_{\G})\in\hat{\Phi}(M(c_{0}(\hat{\HH})))$.\end{proof}
\begin{rem}
For $\mu\in\CzU{\G}_{+}^{*}$, the condition $(\mu\tensor\i)(\WW_{\G})\notin\hat{\Phi}(M(\CzU{\hat{\HH}}))$
from \prettyref{def:sep_prop} implies that $\mu\notin\Phi^{*}(\CzU{\HH}^{*})$,
because if $\mu=\nu\circ\Phi$ for some $\nu\in C_{0}^{\mathrm{u}}(\HH)_{+}^{*}$,
then $(\mu\tensor\i)(\WW_{\G})=\hat{\Phi}((\nu\tensor\i)(\WW_{\HH}))\in\hat{\Phi}(M(\CzU{\hat{\HH}}))$
by \prettyref{eq:quantum_hom_and_dual}. Moreover, if $\G$ is commutative,
the two conditions are equivalent. We do not know whether \prettyref{thm:compact_sep_prop}
holds with this weaker condition as well.
\end{rem}

\subsection{\label{sub:bicrossed_products}Examples arising from the bicrossed
product construction}

A natural way to construct a closed quantum subgroup of a LCQG is
the bicrossed product (see Vaes and Vainerman \citep{Vaes_Vainerman__extensions_LCQGs}).
Let $\G_{1},\G_{2}$ be LCQGs. We say that $(\G_{1},\G_{2})$ is a
\emph{matched pair} \citep[Definition 2.1]{Vaes_Vainerman__extensions_LCQGs}
if it admits a cocycle matching $(\tau,\mathscr{U},\mathscr{V})$,
which means that $\tau:\Linfty{\G_{1}}\tensorn\Linfty{\G_{2}}\to\Linfty{\G_{1}}\tensorn\Linfty{\G_{2}}$
is a faithful, normal, unital $*$-homomorphism and $\mathscr{U}\in\Linfty{\G_{1}}\tensorn\Linfty{\G_{1}}\tensorn\Linfty{\G_{2}}$,
$\mathscr{V}\in\Linfty{\G_{1}}\tensorn\Linfty{\G_{2}}\tensorn\Linfty{\G_{2}}$
are unitaries such that the $*$-homomorphisms 
\[
\a:\Linfty{\G_{2}}\to\Linfty{\G_{1}}\tensorn\Linfty{\G_{2}},\qquad\be:\Linfty{\G_{1}}\to\Linfty{\G_{1}}\tensorn\Linfty{\G_{2}}
\]
given by $\a(y):=\tau(\one\tensor y)$, $y\in\Linfty{\G_{2}}$, and
$\be(x):=\tau(x\tensor\one)$, $x\in\Linfty{\G_{1}}$, satisfy the
following conditions:
\begin{enumerate}
\item $(\a,\mathscr{U})$ is a left cocycle action of $\G_{1}$ on $\Linfty{\G_{2}}$,
that is:
\[
(\i\tensor\a)(\a(y))=\mathscr{U}(\Delta_{1}\tensor\i)(\a(y))\mathscr{U}^{*}\qquad(\forall y\in\Linfty{\G_{2}}),
\]
\[
(\i\tensor\i\tensor\a)(\mathscr{U})(\Delta_{1}\tensor\i\tensor\i)(\mathscr{U})=(\one\tensor\mathscr{U})(\i\tensor\Delta_{1}\tensor\i)(\mathscr{U});
\]

\item $(\sigma\be,\mathscr{V}_{321})$ is a left cocycle action of $\G_{2}$
on $\Linfty{\G_{1}}$, that is:
\[
(\be\tensor\i)(\be(x))=\mathscr{V}(\i\tensor\Delta_{2}^{\op})(\be(x))\mathscr{V}^{*}\qquad(\forall x\in\Linfty{\G_{1}}),
\]
\[
(\be\tensor\i\tensor\i)(\mathscr{V})(\i\tensor\i\tensor\Delta_{2}^{\op})(\mathscr{V})=(\mathscr{V}\tensor\one)(\i\tensor\Delta_{2}^{\op}\tensor\i)(\mathscr{V});
\]

\item $(\a,\mathscr{U})$ and $(\be,\mathscr{V})$ are matched, that is:
\[
\tau_{13}(\a\tensor\i)(\Delta_{2}(y))=\mathscr{V}_{132}(\i\tensor\Delta_{2})(\a(y))\mathscr{V}_{132}^{*}\qquad(\forall y\in\Linfty{\G_{2}}),
\]
\[
\tau_{23}\sigma_{23}(\be\tensor\i)(\Delta_{1}(x))=\mathscr{U}(\Delta_{1}\tensor\i)(\be(x))\mathscr{U}^{*}\qquad(\forall x\in\Linfty{\G_{1}}),
\]
\begin{multline}
(\Delta_{1}\tensor\i\tensor\i)(\mathscr{V})(\i\tensor\i\tensor\Delta_{2}^{\op})(\mathscr{U}^{*})\\
=(\mathscr{U}^{*}\tensor\one)(\i\tensor\tau\sigma\tensor\i)\left[(\be\tensor\i\tensor\i)(\mathscr{U}^{*})(\i\tensor\i\tensor\a)(\mathscr{V})\right](\one\tensor\mathscr{V}).\label{eq:matched_U_V}
\end{multline}

\end{enumerate}
Suppose that such a matched pair is given. For convenience, write
$\H_{i}:=\Ltwo{\G_{i}}$, $i=1,2$, and let $\tilde{W}:=(W_{1}\tensor\one)\mathscr{U}^{*}\in\Linfty{\G_{1}}\tensorn B(\H_{1})\tensorn\Linfty{\G_{2}}$.
Recall that the cocycle crossed product $\G_{1}\presb{\a,\mathscr{U}}{\ltimes}\Linfty{\G_{2}}$
is the von Neumann subalgebra of $B(\H_{1})\tensorn\Linfty{\G_{2}}$
generated by $\a(\Linfty{\G_{2}})$ and $\{(\om\tensor\i\tensor\i)(\tilde{W}):\om\in\Lone{\G_{1}}\}$.
Letting $\H:=\H_{1}\tensor\H_{2}$, define unitaries $W,\hat{W}\in B(\H\tensor\H)$
by 
\[
\hat{W}:=(\be\tensor\i\tensor\i)\bigl[(W_{1}\tensor\one)\mathscr{U}^{*}\bigr](\i\tensor\i\tensor\a)\bigl[\mathscr{V}(\one\tensor\hat{W}_{2})\bigr],\qquad W:=\sigma(\hat{W}^{*}).
\]
By \citep[Theorem 2.13]{Vaes_Vainerman__extensions_LCQGs}, there
is a LCQG $\G$ with $\Linfty{\G}=\G_{1}\presb{\a,\mathscr{U}}{\ltimes}\Linfty{\G_{2}}$,
$\Ltwo{\G}=\H$ and $W$ being its left regular co-representation.
Defining $\tilde{\tau}:=\sigma\tau\sigma$, $\tilde{\mathscr{U}}:=\mathscr{V}_{321}$
and $\tilde{\mathscr{V}}:=\mathscr{U}_{321}$, one checks that $(\tilde{\tau},\tilde{\mathscr{U}},\tilde{\mathscr{V}})$
is a cocycle matching making $(\G_{2},\G_{1})$ into a matched pair.
Its ambient LCQG is, up to flipping from $\H_{2}\tensor\H_{1}$ to
$\H_{1}\tensor\H_{2}$, precisely the dual $\hat{\G}$. In what follows
we use a subscript to indicate that a symbol relates to $\G_{i}$,
$i=1,2$, and a lack of subscript if it relates to $\G$. For instance,
$J_{1}$, $J_{2}$ and $J$ are the modular conjugations of $\Linfty{\G_{1}}$,
$\Linfty{\G_{2}}$ and $\Linfty{\G}$, respectively.

Since $\Delta\circ\a=(\a\tensor\a)\circ\Delta_{2}$ \citep[Proposition 2.4]{Vaes_Vainerman__extensions_LCQGs},
we see that $\hat{\G}_{2}$ is a closed quantum subgroup of $\hat{\G}$
in the sense of Vaes, thus also in the sense of Woronowicz.

It is proved in \citep[Section 3]{Vaes_Vainerman__extensions_LCQGs}
that there is a bijection between (cocycle) bicrossed products and
\emph{cleft extensions} of LCQGs. To elaborate, consider the unitary
\[
Z_{2}:=(J_{1}\tensor\hat{J})(\i\tensor\be)(\hat{W}_{1}^{*})(J_{1}\tensor\hat{J}).
\]
Then the formula $\theta(z):=Z_{2}(\one\tensor z)Z_{2}^{*}$ defines
a map $\theta:\Linfty{\G}\to\Linfty{\hat{\G}_{1}}\tensorn\Linfty{\G}$,
which is an action of $\hat{\G}_{1}^{\op}$ on $\G$. The exactness
of the sequence at $\G$ is manifested by the following characterization
of the fixed-point algebra of $\theta$: 
\begin{equation}
\Linfty{\G}^{\theta}=\a(\Linfty{\G_{2}}).\label{eq:extension_exactness}
\end{equation}
The proof of this is by no means technical: it strongly relies on
the structure of $\G$ and its dual.
\begin{example}
\label{exa:bicrossed_products}Assume henceforth that $\G_{2}$ is
discrete and, denoting by $p$ the central minimal projection in $\Linfty{\G_{2}}$
with $yp=\epsilon_{2}(y)p=py$ for every $y\in\Linfty{\G_{2}}$, that
\begin{gather}
\a(p)=\one_{\Linfty{\G_{1}}}\tensor p,\label{eq:matched_pair_alpha_p}\\
(\i\tensor\epsilon_{2})\be=\i,\label{eq:matched_pair_discrete_beta}\\
(\i\tensor\i\tensor\epsilon_{2})(\mathscr{V})=\one_{\Linfty{\G_{1}}}\tensor\one_{\Linfty{\G_{2}}}=(\i\tensor\epsilon_{2}\tensor\i)(\mathscr{V}).\label{eq:matched_pair_discrete_V}
\end{gather}
Condition \prettyref{eq:matched_pair_alpha_p} means, essentially,
that $\G_{1}$ is ``connected'', while \prettyref{eq:matched_pair_discrete_beta}
and \prettyref{eq:matched_pair_discrete_V} are natural as $\G_{2}$
is discrete (see Vaes and Vergnioux \citep[Definition 1.24]{Vaes_Vergnioux__boundary}
and Packer and Raeburn \citep[Definition 2.1]{Packer_Raeburn__twisted_crs_prd}).

For starters, notice that 
\begin{equation}
(\i\tensor\i\tensor\epsilon_{2})(\mathscr{U})=\one_{\Linfty{\G_{1}}}\tensor\one_{\Linfty{\G_{1}}}.\label{eq:matched_pair_discrete_U}
\end{equation}
Indeed, denote the left-hand side by $U$. Applying the $*$-homomorphism
$\i\tensor\i\tensor\i\tensor\epsilon_{2}$ to \prettyref{eq:matched_U_V}
and using \prettyref{eq:matched_pair_discrete_V}, we obtain 
\[
\mathscr{U}^{*}=\mathscr{U}^{*}(\i\tensor\tau\sigma)\left[(\be\tensor\i)(U^{*})(\i\tensor\i\tensor(\i\tensor\epsilon_{2})\a)(\mathscr{V})\right],
\]
and since $\i\tensor\tau\sigma$ is faithful, 
\[
(\be\tensor\i)(U)=(\i\tensor\i\tensor(\i\tensor\epsilon_{2})\a)(\mathscr{V}).
\]
Applying $\i\tensor\epsilon_{2}\tensor\i$ and using \prettyref{eq:matched_pair_discrete_beta}
and \prettyref{eq:matched_pair_discrete_V}, we get $U=(\i\tensor(\i\tensor\epsilon_{2})\a)(\one)=\one$,
as desired.

We claim that for every $z\in\Linfty{\G}$, 
\[
\Delta(z)(\a(p)\tensor\one_{\Linfty{\G}})=\a(p)\tensor z\quad\implies\quad z\in\a(\Linfty{\G_{2}}).
\]
Indeed, suppose that the assumption is met. Then $\Delta^{\op}(z)(\one_{\Linfty{\G}}\tensor\one_{B(\H_{1})}\tensor p)=z\tensor\one_{B(\H_{1})}\tensor p$
by \prettyref{eq:matched_pair_alpha_p}. From \citep[Lemma 2.3]{Vaes_Vainerman__extensions_LCQGs}
we get 
\begin{multline*}
\Delta^{\op}(z)(\one_{\Linfty{\G}}\tensor\one_{B(\H_{1})}\tensor p)\\
\begin{split} & =(\be\tensor\i\tensor\i)(\tilde{W})\bigl[(\i\tensor\i\tensor\a)(\mathscr{V}(\i\tensor\Delta_{2}^{\op})(z)\mathscr{V}^{*})\bigr](\be\tensor\i\tensor\i)(\tilde{W}^{*})(\one_{\Linfty{\G}}\tensor\one_{B(\H_{1})}\tensor p)\\
 & =(\be\tensor\i\tensor\i)(\tilde{W})\bigl[(\i\tensor\i\tensor\a)(\mathscr{V}(\i\tensor\Delta_{2}^{\op})(z)(\one_{\Linfty{\G}}\tensor p)\mathscr{V}^{*})\bigr](\be\tensor\i\tensor\i)(\tilde{W}^{*})\\
 & =(\be\tensor\i\tensor\i)(\tilde{W})\bigl[(\i\tensor\i\tensor\a)(\mathscr{V}(z\tensor p)\mathscr{V}^{*})\bigr](\be\tensor\i\tensor\i)(\tilde{W}^{*}).
\end{split}
\end{multline*}
By \prettyref{eq:matched_pair_discrete_V} and \prettyref{eq:matched_pair_alpha_p}
we thus have 
\[
\begin{split}\Delta^{\op}(z)(\one_{\Linfty{\G}}\tensor\one_{B(\H_{1})}\tensor p) & =(\be\tensor\i\tensor\i)(\tilde{W})\bigl[(\i\tensor\i\tensor\a)(z\tensor p)\bigr](\be\tensor\i\tensor\i)(\tilde{W}^{*})\\
 & =(\be\tensor\i\tensor\i)(\tilde{W})(z\tensor\one_{B(\H_{1})}\tensor p)(\be\tensor\i\tensor\i)(\tilde{W}^{*}).
\end{split}
\]
The assumption hence implies that 
\[
(\be\tensor\i\tensor\i)(\tilde{W})(z\tensor\one_{B(\H_{1})}\tensor p)(\be\tensor\i\tensor\i)(\tilde{W}^{*})=z\tensor\one_{B(\H_{1})}\tensor p.
\]
Applying $\i\tensor\i\tensor\i\tensor\epsilon_{2}$ to both sides,
we deduce from \prettyref{eq:matched_pair_discrete_U} that 
\[
(\be\tensor\i)(W_{1})(z\tensor\one_{B(\H_{1})})(\be\tensor\i)(W_{1}^{*})=z\tensor\one_{B(\H_{1})}.
\]
Writing $w:=\hat{J}z\hat{J}$ and recalling that $w=R(z^{*})\in\Linfty{\G}$
where $R$ is the unitary antipode of $\G$, the last equation is
equivalent to $\theta(w)=\one\tensor w$, that is, $w\in\Linfty{\G}^{\theta}$.
By \prettyref{eq:extension_exactness}, $R(z^{*})$ belongs to the
image of $\a$. By the von Neumann algebraic version of \citep[Corollary 5.46]{Kustermans_Vaes__LCQG_C_star},
we have $R\circ\a=\a\circ R_{2}$. Therefore $z$ belongs to the image
of $\a$, and the proof is complete.\end{example}
\begin{rem}
The LCQG $\G$ constructed in \prettyref{exa:bicrossed_products}
and its dual are neither necessarily amenable nor necessarily co-amenable
\citep[Theorems 13 and 15]{Desmedt_Quaegebeur_Vaes}.
\end{rem}

\begin{rem}
The last part of the reasoning in \prettyref{exa:bicrossed_products}
uses in an essential way the \emph{exactness} of the short exact sequence
of LCQGs. As mentioned above, bicrossed products are characterized
as \emph{cleft} extensions. By \citep[Propositions 1.22 and 1.24]{Vaes_Vainerman__extensions_LCQGs},
this amounts to the structure of $\Linfty{\G}$ as a cocycle crossed
product. Examining the argument in \prettyref{exa:bicrossed_products},
this structure is used mainly in the simplification of $\Delta^{\op}(z)(\one\tensor\one\tensor p)$.
It is not clear at the moment whether this argument generalizes further,
thus leaving the general case of compact (or, even more generally,
closed) \emph{normal} quantum subgroups (see Vaes and Vainerman \citep{Vaes_Vainerman__low_dim_LCQGs})
open.
\end{rem}

\subsection{\label{sub:E_2}Example: quantum $E(2)$ group}

We prove that the complex unit circle $\mathbb{T}$, as a closed quantum
subgroup of $E(2)$, has the separation property. Considering the
quantum groups $E(2)$ and $\hat{E}(2)$, we essentially follow the
notation of Jacobs \citep{Jacobs_phd_thesis} although it does not
always agree with ours; further details can be found there. Fix $0<\mu<1$.
Set $\R^{\mu}:=\{\mu^{k}:k\in\Z\}$, $\overline{\R}^{\mu}:=\R^{\mu}\cup\left\{ 0\right\} $,
$\R(\mu^{1/2}):=\{\mu^{k/2}:k\in\Z\}$ and $\overline{\R}(\mu^{1/2}):=\R(\mu^{1/2})\cup\left\{ 0\right\} $.

The following is taken from \citep[Section 2.3]{Jacobs_phd_thesis}.
Let $\left(e_{k}\right)_{k\in\Z}$ be an orthonormal basis of $\ell_{2}(\Z)$.
Denote by $s$ the unitary operator over $\ell_{2}(\Z)$ which is
the shift given by $se_{k}:=e_{k+1}$, $k\in\Z$. Denote by $m$ the
strictly positive (unbounded) operator over $\ell_{2}(\Z)$ that acts
on its core $\linspan\left\{ e_{k}:k\in\Z\right\} $ by $me_{k}:=\mu^{k}e_{k}$,
$k\in\Z$.

Set $\H:=\ell_{2}(\Z)\tensor\ell_{2}(\Z)$ and $e_{k,l}:=e_{k}\tensor e_{l}$
for $k,l\in\Z$. Consider the unbounded operators over $\H$ defined
by 
\[
a:=m^{-1/2}\tensor m,\qquad b:=m^{1/2}\tensor s.
\]
Then $a$ is strictly positive, $b$ has polar decomposition $b=u\left|b\right|$
with $u:=\one\tensor s$ and $\left|b\right|=m^{1/2}\tensor\one$,
and $\sigma(a)=\overline{\R}(\mu^{1/2})=\sigma(b)$. Since $a,\left|b\right|$
commute, they have a joint Borel functional calculus. As observed
in \citep[Remark 2.5.20]{Jacobs_phd_thesis}, the joint continuous
functional calculus of $a,\left|b\right|$ is determined by the values
of the functions on $E:=\{(p,q)\in\R(\mu^{1/2})\times\overline{\R}(\mu^{1/2}):pq\in\overline{\R}^{\mu}\}$.
Similarly, as $\left|b\right|$, just like $a$, is injective, the
joint Borel functional calculus of $a,\left|b\right|$ is determined
by $F:=\{(p,q)\in\R(\mu^{1/2})\times\R(\mu^{1/2}):pq\in\R^{\mu}\}$.
Writing $\mathbb{B}(F)$ for the algebra of all bounded complex-valued
functions over $F$, we get an injection $\mathbb{B}(F)\ni g\mapsto g(a,\left|b\right|)\in B(\H)$.

The operator $W\in B(\H\tensor\H)$ is the unitary that satisfies
\begin{equation}
((\om_{e_{k,l},e_{p,q}}\tensor\i)(W))e_{m,n}=B(q-l,k-l-n+1)\delta_{k,p}e_{m-k+2q,n-k+l+q}\qquad(\forall k,l,p,q,m,n\in\Z),\label{eq:E_2__W}
\end{equation}
where $\left(B(k,n)\right)_{k,n\in\Z}$ are special scalars in the
complex unit disc.

The right, resp.~left, leg of $W$ norm-spans a $C^{*}$-algebra
$A$, resp.~$\hat{A}$, which is the reduced $C^{*}$-algebra underlying
the LCQG $E(2)$, resp.~$\hat{E}(2)$, and $W\in M(\hat{A}\tensormin A)$
\citep[Sections 2.4, 2.5]{Jacobs_phd_thesis}. The co-multiplications
$\Delta:A\to M(A\tensormin A)$, resp.~$\hat{\Delta}:\hat{A}\to M(\hat{A}\tensormin\hat{A})$
of $E(2)$, resp.~$\hat{E}(2)$, is given by $\Delta(x)=W(x\tensor\one)W^{*}$
for $x\in A$, resp.~$\hat{\Delta}(y):=W^{*}(\one\tensor y)W$ for
$y\in\hat{A}$. The duality relation between $E(2)$ and $\hat{E}(2)$
is opposite: $\hat{E}(2)=\widehat{E(2)}^{\op}$ \citep[Proposition 2.8.21]{Jacobs_phd_thesis},
but since $\mathbb{T}$ is commutative, that is meaningless for our
purposes.

The unbounded operators $a,a^{-1},b$ are affiliated with $\hat{A}$
in the sense of $C^{*}$-algebras, and $a$ is ``group like'', that
is, $\hat{\Delta}(a)=a\tensor a$, where the left-hand side is interpreted
as a nondegenerate $*$-homomorphism acting on an affiliated element.
This makes $\mathbb{T}$ a closed quantum subgroup of $E(2)$: identifying
$\ell_{\infty}(\Z)\cong\left\{ f(a):f\in C_{b}(\R(\mu^{1/2}))\right\} $
(recall that $a$ is injective!), the embedding $\gamma:\ell_{\infty}(\Z)\hookrightarrow M(\hat{A})$
is given by mapping $g\in\ell_{\infty}(\Z)$ to $f(a)$, where $f(\mu^{k/2}):=g(k)$,
$k\in\Z$ \citep[Subsection 2.8.5]{Jacobs_phd_thesis}. Denote by
$p$ the projection $k\mapsto\delta_{k,0}$ in $\ell_{\infty}(\Z)$.
Then $\gamma(p)$ is the projection onto $\left\{ e_{2l,l}:l\in\Z\right\} $.

To establish the separation property, consider all $y\in M(\hat{A})$
satisfying $\hat{\Delta}(y)(\gamma(p)\tensor\one)=(\gamma(p)\tensor y)$.
This means that $\one\tensor y$ commutes with $W(\gamma(p)\tensor\one)$,
or equivalently, that $y$ commutes with $(\om_{\z,\eta}\tensor\i)(W)$
for every $\z\in\Img\gamma(p)$ and $\eta\in\H$. Substituting $q-l$
for $t$ in \prettyref{eq:E_2__W}, this amounts to $y$ commuting
with each of the operators $x_{l,t}\in B(\H)$, $l,t\in\Z$, given
by $x_{l,t}e_{m,n}:=B(t,l-n+1)e_{m+2t,n+t}$ for $m,n\in\Z$.

Let $l,t\in\Z$. Clearly, $x_{l,t}$ commutes with $a$. For $m,n\in\Z$
and $s\in\R$, 
\[
\begin{split}\left|b\right|^{is}x_{l,t}e_{m,n} & =B(t,l-n+1)\left|b\right|^{is}e_{m+2t,n+t}=\mu^{is(m+2t)/2}B(t,l-n+1)e_{m+2t,n+t},\\
x_{l,t}\left|b\right|^{is}e_{m,n} & =\mu^{ism/2}x_{l,t}e_{m,n}=\mu^{ism/2}B(t,l-n+1)e_{m+2t,n+t},
\end{split}
\]
so that $\left|b\right|^{is}x_{l,t}=\mu^{ist}x_{l,t}\left|b\right|^{is}$
for every $s\in\R$, or formally $\left|b\right|x_{l,t}=\mu^{t}x_{l,t}\left|b\right|$.
This implies that for every $g\in\mathbb{B}(F)$ we have 
\begin{equation}
g(a,\left|b\right|)x_{l,t}=x_{l,t}g_{t}(a,\left|b\right|),\label{eq:E_2_g_x}
\end{equation}
where $g_{t}\in\mathbb{B}(F)$ is defined by $g_{t}(\a,\be):=g(\a,\mu^{t}\be)$.
Moreover, for $k,m,n\in\Z$,
\begin{equation}
\begin{split}u^{k}x_{l,t}e_{m,n} & =B(t,l-n+1)u^{k}e_{m+2t,n+t}=B(t,l-n+1)e_{m+2t,n+t+k},\\
x_{l,t}u^{k}e_{m,n} & =x_{l,t}e_{m,n+k}=B(t,l-n-k+1)e_{m+2t,n+t+k.}
\end{split}
\label{eq:E_2__u_x}
\end{equation}

\begin{lem}
\label{lem:E_2_summand_comm_x_lt}Let $g\in\mathbb{B}(F)$ and $k\in\Z$.
Assume that $u^{k}g(a,\left|b\right|)$ commutes with the operators
$\left(x_{l,t}\right)_{l,t\in\Z}$. If $k\neq0$, then $g(a,\left|b\right|)=0$;
if $k=0$, then $g$ is the restriction of $h\tensor\one$ for some
$h\in\mathbb{B}(\R(\mu^{1/2}))$.\end{lem}
\begin{proof}
Both cases will use the following computation. Let $t,m,n\in\Z$.
Since $\C e_{m,n}$ is invariant under both $a$ and $\left|b\right|$,
it is invariant under $g(a,\left|b\right|)$ and $g_{t}(a,\left|b\right|)$.
Let $\gamma,\gamma_{t}\in\C$ be such that $g(a,\left|b\right|)e_{m,n}=\gamma e_{m,n}$
and $g_{t}(a,\left|b\right|)e_{m,n}=\gamma_{t}e_{m,n}$. By assumption,
for all $l\in\Z$ we have $x_{l,t}u^{k}g(a,\left|b\right|)=u^{k}g(a,\left|b\right|)x_{l,t}=u^{k}x_{l,t}g_{t}(a,\left|b\right|)$
from \prettyref{eq:E_2_g_x}, so using \prettyref{eq:E_2__u_x},
\[
x_{l,t}u^{k}g(a,\left|b\right|)e_{m,n}=\gamma x_{l,t}u^{k}e_{m,n}=\gamma B(t,l-n-k+1)e_{m+2t,n+t+k}
\]
is equal to
\[
u^{k}x_{l,t}g_{t}(a,\left|b\right|)e_{m,n}=\gamma_{t}u^{k}x_{l,t}e_{m,n}=\gamma_{t}B(t,l-n+1)e_{m+2t,n+t+k},
\]
that is, 
\begin{equation}
\gamma B(t,l-n-k+1)=\gamma_{t}B(t,l-n+1).\label{eq:E_2_summand_comm_x_lt}
\end{equation}

Suppose that $k=0$. Let $t,m,n\in\Z$ and let $\gamma,\gamma_{t}\in\C$
be as above. Then for every $l\in\Z$, we have $\gamma B(t,l-n+1)=\gamma_{t}B(t,l-n+1)$
from \prettyref{eq:E_2_summand_comm_x_lt}. Choosing $l$ such that
$B(t,l-n+1)\neq0$, which is possible by \citep[Corollary A.11]{Jacobs_phd_thesis},
we get $\gamma=\gamma_{t}$. As $m,n$ were arbitrary, we deduce that
$g(a,\left|b\right|)=g_{t}(a,\left|b\right|)$, hence $g=g_{t}$.
By the definition of $F$, as $t$ was arbitrary, $g$ is of the form
$h\tensor\one$.

Suppose that $k\neq0$. Since $\left(B(t,0)\right)_{t\in\Z}$ are
the Fourier coefficients of a non-constant function \citep[Definition A.4]{Jacobs_phd_thesis},
we can fix $0\neq t\in\Z$ with $B(t,0)\neq0$. Assuming that $g(a,\left|b\right|)\neq0$,
fix $m,n\in\Z$ such that $g(a,\left|b\right|)e_{m,n}\neq0$. Let
$\gamma,\gamma_{t}\in\C$ be as above; then $\gamma\neq0$. Replacing
$l-n+1$ by $l$ in \prettyref{eq:E_2_summand_comm_x_lt} for convenience,
we get $\gamma B(t,l-k)=\gamma_{t}B(t,l)$ for all $l\in\Z$, and
in particular, $\gamma_{t}\neq0$ (take $l=k$). Hence $B(t,sk)=(\gamma_{t}/\gamma)^{-s}B(t,0)$
for all $s\in\Z$. From \citep[Proposition A.9]{Jacobs_phd_thesis},
since $t\neq0$, we have $B(t,l)\xrightarrow[\left|l\right|\to\infty]{}0$,
a contradiction.
\end{proof}
Let $\hat{M}$ be the strong closure of $\hat{A}$ in $B(\H)$. We
need a certain expansion of elements of $\hat{M}$.
\begin{lem}
\label{lem:E_2_vN_alg_structure}Every $y\in\hat{M}$ possesses a
(unique) sequence of functions $\left(g_{k}\right)_{k\in\Z}$ in $\mathbb{B}(F)$
such that 
\[
y=\mathrm{strong-}\lim_{N\to\infty}\sum_{k=-N}^{N}(1-\frac{\left|k\right|}{N+1})u^{k}g_{k}(a,\left|b\right|)
\]
\end{lem}
\begin{proof}
For each $\l\in\mathbb{T}$, define a unitary $w_{\l}\in B(\ell_{2}(\Z))$
by $w_{\l}(e_{l}):=\l^{l}e_{l}$ ($l\in\Z$), and a unitary $W_{\l}\in B(\H)$
by $W_{\l}:=1\tensor w_{\l}$. Then $W_{\l}$ commutes with $a,\left|b\right|$
and $W_{\l}uW_{\l}^{*}=\l u$. For every $k\in\N$ and $g_{k}\in\mathbb{B}(F)$
we thus get 
\begin{equation}
\mathrm{Ad}(W_{\l})(u^{k}g_{k}(a,\left|b\right|))=\l^{k}u^{k}g_{k}(a,\left|b\right|).\label{eq:W_lambda_trick}
\end{equation}
Given $n\in\Z$, define the ``Fourier coefficient'' contraction
$\Upsilon_{n}\in B(B(\H))$ by 
\[
\Upsilon_{n}(y):=\frac{1}{2\pi}\int_{\mathbb{T}}\l^{-n}\mathrm{Ad}(W_{\l})(y)\left|\mathrm{d}\l\right|\qquad(y\in B(\H)),
\]
where the integral converges strongly. The operator $\Upsilon_{n}$
is continuous in the bounded strong operator topology. Letting $\left\{ K_{N}\right\} _{N=1}^{\infty}$
denote Fej\'{e}r's kernel, we have, for $y\in B(\H)$ and $N\in\N$,
\[
\sum_{n=-N}^{N}(1-\frac{\left|n\right|}{N+1})\Upsilon_{n}(y)=\int_{\mathbb{T}}\frac{1}{2\pi}K_{N}(\l)\mathrm{Ad}(W_{\l})(y)\left|\mathrm{d}\l\right|.
\]
Thus, the sequence $\bigl\{\sum_{n=-N}^{N}(1-\frac{\left|n\right|}{N+1})\Upsilon_{n}(y)\bigr\}_{N=1}^{\infty}$
is bounded by $\left\Vert y\right\Vert $, and it converges strongly
to $y$.

On account of \prettyref{eq:W_lambda_trick}, if $y$ has the form
$\sum_{k=-N}^{N}u^{k}g_{k}(a,\left|b\right|)$ then $\Upsilon_{n}(y)=u^{n}g_{n}(a,\left|b\right|)$
for $-N\leq n\leq N$ and $0$ otherwise. Every element $y$ of $\hat{M}$
is the strong limit of a bounded net $\left(y_{i}\right)$ of elements
of the form $y_{i}=\sum_{k\in\Z}u^{k}g_{ki}(a,\left|b\right|)$, where
$g_{ki}\neq0$ for only finitely-many values of $k$ for every $i$
\citep[Theorem 2.5.21]{Jacobs_phd_thesis}. Consequently, $\Upsilon_{n}(y_{i})=u^{n}g_{ni}(a,\left|b\right|)\to\Upsilon_{n}(y)$
strongly for all $n$. As $u$ is unitary, we infer that the net $(g_{ni}(a,\left|b\right|))_{i}$
converges strongly for all $n$, necessarily to $g_{n}(a,\left|b\right|)$
for some $g_{n}\in\mathbb{B}(F)$. By the foregoing, $y=\lim_{N}\sum_{n=-N}^{N}(1-\frac{\left|n\right|}{N+1})u^{n}g_{n}(a,\left|b\right|)$
strongly. For uniqueness, have $\Upsilon_{n}$ act on both sides of
the equation.\end{proof}
\begin{lem}
\label{lem:E_2_comm_x_lt}Let $y\in\hat{M}$, and let $\left(g_{k}\right)$
be the functions corresponding to $y$ as in \prettyref{lem:E_2_vN_alg_structure}.
If $y$ commutes with all the operators $x_{l,t}\in B(\H)$, $l,t\in\Z$,
then so does $u^{k}g_{k}(a,\left|b\right|)$ for every $k\in\Z$.\end{lem}
\begin{proof}
For $m,n\in\Z$, denote by $p_{m,n}$ the projection of $\H$ onto
$\C e_{m,n}$. Fix $l,t\in\Z$. Clearly, $g(a,\left|b\right|)$ commutes
with $p_{m,n}$ for every $g\in\mathbb{B}(F)$, $up_{m,n}=p_{m,n+1}u$
and $x_{l,t}p_{m,n}=p_{m+2t,n+t}x_{l,t}$. By assumption, we have
$\lim_{N\to\infty}\sum_{k=-N}^{N}(1-\frac{\left|k\right|}{N+1})x_{l,t}u^{k}g_{k}(a,\left|b\right|)=\lim_{N\to\infty}\sum_{k=-N}^{N}(1-\frac{\left|k\right|}{N+1})u^{k}g_{k}(a,\left|b\right|)x_{l,t}$,
both limits being in the strong operator topology. Fix $k_{0}\in\Z$.
For every $m,n\in\Z$, we get 
\begin{multline*}
\lim_{N\to\infty}\sum_{k=-N}^{N}(1-\frac{\left|k\right|}{N+1})p_{m+2t,n+t+k_{0}}x_{l,t}u^{k}g_{k}(a,\left|b\right|)p_{m,n}\\
=\lim_{N\to\infty}\sum_{k=-N}^{N}(1-\frac{\left|k\right|}{N+1})p_{m+2t,n+t+k_{0}}u^{k}g_{k}(a,\left|b\right|)x_{l,t}p_{m,n}.
\end{multline*}
As a result, with $z:=x_{l,t}u^{k_{0}}g_{k_{0}}(a,\left|b\right|)-u^{k_{0}}g_{k_{0}}(a,\left|b\right|)x_{l,t}$,
we have $zp_{m,n}=p_{m+2t,n+t+k_{0}}zp_{m,n}=0$. Summing over all
$m,n\in\Z$ we get the desired commutation relation.
\end{proof}
We are now ready to prove that $\mathbb{T}$ has the separation property
in $E(2)$. If $y\in\hat{M}$ with corresponding functions $\left(g_{k}\right)$
as in \prettyref{lem:E_2_vN_alg_structure} commutes with all the
operators $\left(x_{l,t}\right)_{l,t\in\Z}$, then by \prettyref{lem:E_2_comm_x_lt},
$u^{k}g_{k}(a,\left|b\right|)$ commutes with $\left(x_{l,t}\right)_{l,t\in\Z}$
for every $k\in\Z$. \prettyref{lem:E_2_summand_comm_x_lt} implies
that $g_{k}(a,\left|b\right|)=0$ for $k\neq0$ and that $g_{0}=h\tensor\one$
for a suitable $h$. This precisely means that $y$ is a function
of $a$, namely $y\in\Img\gamma$. So we established \prettyref{eq:compact_sep_prop_condition}
in our setting, and the proof is complete.

\section*{Acknowledgments}

We thank Matthew Daws for his interest and helpful comments. We are
grateful to the referee for carefully reading the paper and making
very useful suggestions.

\bibliographystyle{amsplain}
\bibliography{LCQGs_PosDef}

\end{document}